\documentclass[11pt]{article}
\usepackage{hyperref}
\usepackage[T1]{fontenc}
\usepackage[utf8]{inputenc}
\usepackage{newcent}
\usepackage{enumerate}
\usepackage{changes}
\usepackage{hyperref}
\usepackage{amsmath}
\usepackage{amsfonts}
\usepackage{amssymb}
\usepackage{amsthm}
\usepackage{newlfont}
\usepackage{mathtools}
\usepackage[top=1.5in, bottom=1.5in, left=1.3in, right=1.3in]{geometry}

\usepackage{authblk}

\theoremstyle{plain}
\newtheorem{theorem}{Theorem}[section]
\newtheorem{proposition}[theorem]{Proposition}

\newtheorem{lemma}[theorem]{Lemma}

\newtheorem{assumption}[theorem]{Assumption}

\theoremstyle{definition}
\newtheorem{definition}[theorem]{Definition}
\newtheorem*{claim}{Claim}

\newtheorem{example}{Example}[section]
\newtheorem{remark}[theorem]{Remark}

\newcommand{\RR}{\mathbb{R}}
\newcommand{\NN}{\mathbb{N}}
\newcommand{\ZZ}{\mathbb{Z}}

\newcommand{\cC}{\mathcal{C}}
\newcommand{\drm}{\mathrm{d}}
\newcommand{\euler}{\mathrm{e}}

\DeclareMathOperator{\supp}{supp}

\DeclareMathOperator{\Deg}{Deg}

\DeclareMathOperator{\spec}{spec}
\DeclareMathOperator{\av}{av}
\DeclareMathOperator{\arsinh}{arsinh}

\DeclareMathOperator{\Eins}{\mathbf{1}}

\let\oldint\int
\renewcommand{\int}{\oldint\limits}

\newcommand{\Hmm}[1]{\leavevmode{\marginpar{\tiny%
			$\hbox to 0mm{\hspace*{-0.5mm}$\leftarrow$\hss}%
			\vcenter{\vrule depth 0.1mm height 0.1mm width \the\marginparwidth}%
			\hbox to
			0mm{\hss$\rightarrow$\hspace*{-0.5mm}}$\\\relax\raggedright #1}}}

\begin{document}

\title{Gaussian upper bounds for heat kernels on graphs with unbounded geometry
}
\author{Matthias Keller\thanks{matthias.keller@uni-potsdam.de} }
\author{Christian Rose\thanks{christian.rose@uni-potsdam.de}}
\affil[]{Institut f\"ur Mathematik, Universit\"at Potsdam, 		14476  Potsdam, Germany}
\date{}
\maketitle
\begin{abstract}
We prove large-time Gaussian upper bounds for continuous-time heat kernels of Laplacians on graphs with unbounded geometry. Our estimates hold for centers of large balls satisfying a Sobolev inequality and volume doubling. Distances are measured with respect to an intrinsic metric with finite distance balls and finite jump size. The Gaussian decay is given by Davies' function which is natural and sharp in the graph setting.  Furthermore, we find a new polynomial correction term which does not blow up at zero. Although our main focus is on unbounded Laplacians, the results are new even for the normalized Laplacian. In the case of unbounded vertex degree or degenerating measure, the estimates are affected by new error terms reflecting the unboundedness of the geometry. 
	\\
	
\end{abstract}
\section{Introduction}
First upper bounds for fundamental solutions of parabolic equations go back to the seminal work by Nash, Moser, and Aronson on diverence form operators on $\RR^n$, \cite{Nash-58,Moser-64,Aronson-67}.
Since the celebrated work by Li and Yau \cite{LiYau-86}, geometric and analytic properties of manifolds yielding Gaussian upper heat kernel bounds have been explored intensively in numerous articles including 
\cite{Davies-87,Varopoulos-89,SaloffCoste-92, Grigoryan-94, Carron-96}. Corresponding results on strongly local
Dirichlet spaces have been obtained by Sturm \cite{Sturm-95} and since then in great variety, among them  
\cite{CarlenKS-87,BarlowGK-12, BarlowBCK-09}.\\
On weighted graphs with bounded geometry, bounds for continuous-time heat kernels have been studied since the seminal work of Davies and Pang, \cite{Davies-93,Pang-93}. Among many other milestones over the last decades, we mention here the fundamental work of Delmotte on characterizations of upper and lower bounds for heat kernels, \cite{Delmotte-99}, see also
\cite{HebischSC-93, CoulhonG-98, MetzgerS-00, BarlowCK-05, BauerHuaYau-15, BauerHLLMY-15, Barlow-17,GLLY-19,ChenKW-20}.\\
Our focus lies on heat kernel upper bounds of Laplacians on graphs with unbounded geometry. The  Davies--Gaffney--Grigor'yan lemma, i.e., an $\ell^2$-heat kernel bound, was proven for possibly unbounded graph Laplacians in \cite{BauerHuaYau-17} by Bauer, Hua, and Yau in terms of intrinsic metrics. Folz, \cite{Folz-11}, derived off-diagonal Gaussian upper bounds in terms of intrinsic metrics from on-diagonal estimates via Grigor'yan's two-point method. Such bounds have also been obtained for elliptic operators with possibly unbounded edge weights but bounded means on combinatorial graphs in \cite{MourratOtto-16,AndresDS-16}. 
In a remarkable article, Barlow and Chen, 
\cite{BarlowChen-16}, extended Delmotte's work on normalized Laplacians, i.e., bounded Laplacians on graphs, to unbounded combinatorial geometry. 
\\
Of course, this leads to the question: “Why study graphs of unbounded geometry
at all?“ On the one hand, a general result which does not include any superfluous
assumptions is always desirable from a pure mathematical point of view. This includes the opportunity
to handle diffusion processes on graphs beyond constant speed in a unified framework.
Moreover, from the perspective of potential applications graphs with unbounded vertex degree are quite natural. Indeed, many real life networks (such as social networks, the internet, citation
networks, pandemic transmission, etc.) are distinguished by the fact that they have few vertices
with very high degree (hubs, super spreaders, etc.) while the majority of vertices have a rather
low degree.
\\

In the present article we will prove Gaussian upper bounds for heat kernels of Laplacians on weighted graphs with unbounded geometry. We assume a Sobolev inequality and volume doubling only in large balls encoded via intrinsic metrics and obtain sharp Gaussian bounds for all large times. The exponential factor is given by the graph-specific Gaussian found by Davies \cite{Davies-93} and includes the long-time asymptotics given by the bottom of the spectrum of the Laplacian. The sharpness can be seen from the correct long-time behavior in terms of the bottom of the spectrum and the sharp estimates for the heat kernel on the integers obtained by Pang. The unboundedness of the geometry enters the estimates for non-normalized Laplacians via a new correction function including means of the vertex degree and inverse of the measure. Moreover, we obtain a new polynomial correction term which is better than the corresponding term on manifolds.\\ 

A crucial ingredient are intrinsic metrics which are a powerful concept allowing to deal with unbounded Laplacians on graphs. They have already been used to prove a great variety of results, \cite{BauerKH-13, HaeselerKW-13, Folz-14b, Folz-14, BauerKW-15,Keller-15, HuangKS-20, KellerLW-21}. Originally they have been introduced in Sturm's fundamental work on strongly local Dirichlet forms in \cite{Sturm-94} and for general regular Dirichlet forms in \cite{FrankLenzWingert-14}.  
While first ideas towards of intrinsic metrics on graphs can be found in \cite{Davies-93a}, where Davies already noticed that the combinatorial metric should be replaced in heat kernel bounds, they appear first explicitly  in \cite{Folz-11,GrigoryanHuangMasamune-12}. 
\\
Davies also observed that the general lack of chain rule for the gradient on graphs yields a Gaussian exponential which deviates from the one on manifolds, \cite{Davies-93}. This disparity is also reflected by the short-time behavior of the heat kernel which is different from the well-established Varadhan asymptotics on manifolds, \cite{KellerLVW-15}. However, the long term behavior of Davies' Gaussian corresponds asymptotically to the one on manifolds.  Our results show that this is also true in the case of unbounded geometry.\\
Our bounds on the heat kernel depend on new geometric terms which neither appeared so far in the study of bounded Laplacians nor elliptic operators on combinatorial graphs. 

\subsection{The set-up}

The terminology used here follows mostly \cite{KellerLW-21}. Let  $X$ be a countable set, $m\colon X\to(0,\infty)$ a measure on $X$ of full support, and $b\colon X\times X\to [0,\infty)$ be a connected symmetric graph over $(X,m)$. We assume further that $b$ is \emph{locally finite}, i.e.~the sets $ \{y\in X\mid b(x,y)>0\} $ are finite for all $ x\in X $. The space of all real-valued continuous functions on $X$ is denoted by $\cC(X)$. Furthermore, $\Vert\cdot\Vert_p$ denotes the norm in the space $\ell^p(X,m)$, $p\in [1,\infty]$. For $f\in\cC(X)$ and $ x,y\in  X $, we abbreviate $$ \nabla_{xy}f:=f(x)-f(y) ,$$ and let $\Delta\colon \cC(X)\to\cC(X)$ be the Laplace operator given by
\[
\Delta f(x)=\frac1{m(x)}\sum_{y\in X} b(x,y)\nabla_{xy}f, 
\] 
where $ f\in \cC(X) $ and $  x\in X  $ and we interpret the right-hand side as divergence form operator.
By local finiteness, $ \Delta $ maps the compactly supported functions $ \cC_{c}(X) $ into itself and therefore the restriction of $ \Delta $ to $ \cC(X) $ is a symmetric operator on $ \ell^{2}(X,m) $.  By slight  abuse of notation we also denote the Friedrichs realization of $\Delta$ in $\ell^2(X,m)$ by   $\Delta$ and set
\[
\Lambda:=\inf\spec(\Delta).
\]
The heat semigroup of $ \Delta $ acting on $ \ell^{2}(X,m) $ will be denoted by $(P_t)_{t\geq 0}$, i.e., $P_t=\euler^{-t\Delta}$, $t\geq 0$. By discreteness of $ X $ the heat semigroup has a kernel $p\colon [0,\infty)\times X\times X\to[0,\infty)$, called the \emph{heat kernel},
which satisfies for all $ f\in \ell^2(X,m),  x\in X, t\geq 0 $
\[
P_tf(x)=\sum_{x\in X} m(x) p_t(x,y)f(y) .
\]
Furthermore, $ u=P_{t}f $ solves the heat equation for $ f\in \ell^{2}(X,m) $

\[
\frac{d}{dt}u=-\Delta u, \quad u(0,\cdot)=f.
\] 

Our heat kernel estimates will be formulated in terms of an \emph{intrinsic metric}  with respect to $b$ on $(X,m)$ which is a pseudo-metric $\rho\colon X\times X\to[0,\infty)$  such that 
\[
\sum_{y\in X} b(x,y)\rho^2(x,y)\leq m(x),
\]
for all  $ x\in X $.
In the following, $\rho$ will always be a non-trivial intrinsic (pseudo-) metric. The jump size is given by
\[
S:=\sup\{\rho(x,y)\colon x,y\in X, b(x,y)>0\}>0.
\]
\begin{example}\label{ex:intrinsic}For a given graph, the pseudo-metric,
	\begin{align*}
		\rho(x,y)=\inf_{x=x_{0}\sim \ldots\sim x_{k}=y}\sum_{j=0}^{k-1}\left(\frac{1}{\Deg(x_{j})}\wedge \frac{1}{\Deg(x_{j+1})}\right)^{1/2}\wedge S, 
	\end{align*}
for $ x,y\in X $
	always gives an intrinsic metric with bound on the  jump size by $S$. 
	Note that the combinatorial distance is intrinsic if and only if the corresponding Laplacian is bounded, cf. \cite[Lemma~11.22]{KellerLW-21}.
\end{example}

For $R\geq 0$ and $x\in X$, we denote the distance balls with respect to $ \rho $ by
$$B(R):=B_x(R)
:=\{y\in X\colon \rho(x,y)\leq R\}.$$
A standing assumption on the intrinsic metric is the following.
\begin{assumption} The distance balls with respect to the intrinsic metric are compact and the jump size is finite.
\end{assumption}


For $ x\in X $ and $ f\in\cC(X) $ we define 
\[
\vert \nabla f\vert(x):=\Bigg(\frac{1}{m(x)}\sum_{y\in X}b(x,y) (\nabla_{xy}f)^2\Bigg)^\frac{1}{2}.
\]
The combinatorial interior of a set $ A\subset X $
will be denoted by 
\begin{align*}
	A^{\circ}=\{x\in A\colon b(x,y)=0 \mbox{ for all }y\in X\setminus A \}.
\end{align*}

The results presented in this article are obtained by assuming the classical Sobolev inequality and volume doubling assumptions in terms of intrinsic metrics. In contrast to the typical assumptions on manifolds  we require these properties  on  annuli with positive inradius instead of  small balls. 
\begin{definition}\label{def:sobvol} Let  $x\in X$, $R_2\geq R_1\geq 0$, $n>2$, and $d>0$.
\begin{enumerate}[(i)]
\item The \emph{Sobolev inequality} $S(n,R_1,R_2)$ holds in $x$, if there is a constant $C_S>0$ such that for all $R\in[R_1,R_2]$, $u\in\cC(X)$, $\supp u\subset B(R)^\circ$, we have
\begin{equation*}
\frac{m(B(R))^{\frac{2}{n}}}{C_SR^2}
\Vert u\Vert_{\frac{2n}{n-2}}^2
\leq \Vert \vert\nabla u\vert\Vert_2^2+\frac{1}{R^2}\Vert u\Vert_2^2.
\end{equation*}
We abbreviate $S(n,R_1):=S(n,R_1,R_1)$.
\item The \emph{volume doubling} property $V(d,R_1,R_2)$ is satisfied in $x$ if there exists $C_D>0$ such that 
\[
m(B(r_2))\leq C_D \left(\frac{r_2}{r_1}\right)^d m(B(r_1)),\qquad R_1\leq r_1\leq r_2\leq R_2.
\]

\item 
The property $SV(R_1,R_2)=SV(d,n,R_1,R_2)$  holds in $x$ if the Sobolev inequality $S(n,R_1,R_2)$ and volume doubling $V(d,R_1,R_2)$ hold.
\end{enumerate}
\end{definition}

\begin{remark}
We can replace $V(d,R_1,R_2)$ by the 
property $V^\ast(d,R_1,R_2)$:
\[
m(B(2r))\leq C_D^\ast m(B(r)), \qquad r\in[R_1,R_2].
\]
$V(d,R_1,R_2)$ implies $V^\ast(d,R_1,R_2)$ and  $V^\ast(d,R_1,R_2)$ implies  $V(d,R_1,R_2/2)$, however,  with different constants.
Due to this asymmetry the assumption $V(d,R_1,R_2)$ is more natural.
\end{remark}

\subsection{Three special cases of the main result}

Our main result is Theorem~\ref{thm:main1sturm} which shows Gaussian upper heat kernel bounds for all Laplacians in all vertices satisfying $SV(R_1,R_2)$ for  $0 <R_0\leq R_1\leq R_2$ and a certain $R_0$. The heat kernel bounds contain functions depending on $\ell^p$-means of the weighted vertex degree and inverted measure for $p\in(1,\infty]$. 
\\
In this section we present exemplarily special cases in terms of the measure $ m $ and compare them to results in the literature. For the proofs see Section~\ref{section:main}.
\\

As initially observed by Davies, \cite{Davies-93}, instead of the Gaussian $\euler^{-r^2/4t}$ known from manifolds, for graphs the function $\euler^{-\zeta_S(r,t)}$ with 
  \[
\zeta_S(r,t):=\frac{1}{S^2}\left(r S \arsinh\left(\frac{r S}{t}\right)+t-\sqrt{t^2+r^2S^2}\right),
\]
for $ r\geq 0$, $t>0$, appears, where $ S $ is the jump size of the intrinsic metric.

Let $\deg(x):=\sum_{y\in X}b(x,y)$, $x\in X$. Our first choice is the so-called normalized Laplacian, what refers to the choice
$m=\deg$. The combinatorial distance $\rho_{\mathrm{c}}$ is an intrinsic metric with jump size $1$ for this choice of the measure. For this model we have the following. 
\begin{theorem}[Normalizing measure]\label{thm:norm} Let $m=\deg$ and 
$d>0$, $n>2$, 
$R/2\geq  r\geq 42$, 
and $Y\subset X$. Assume 
 that for all $x\in Y$ we have  $SV(r,R)$. There exist $C_0,t_0>0$ such that for all $x,y\in Y$
 and $t\geq t_0$ we have
\begin{align*}
p_{t}(x,y)
\leq 
C_0
\frac{\left(1\vee \sqrt{t^2+\rho_{\mathrm{c}}(x,y)^2}-t\right)^{\frac{n}{2}}}
 {\sqrt{m(B_{x}(\sqrt {t}\wedge R))m(B_{y}(\sqrt {t}\wedge R))}}
\euler^{-\Lambda (t-t\wedge R^2)-\zeta_1\left(\rho_{\mathrm{c}}(x,y),t\right)}.
\end{align*}
\end{theorem}

The function $\zeta_S$ appears naturally in the graph setting, \cite{Davies-93,Pang-93,Delmotte-99,BauerHuaYau-15,BauerHuaYau-17}. 
We have 
\[
-\zeta_1\left(r,t\right)\sim -\frac{r^2}{2t}, \quad t\to \infty,
\]
i.e., the limit of the quotient is one, justifying the name Gaussian behavior. 

Our results are sharp in a certain sense. 
First of all, for the normalized Laplacian on $\ZZ$, 
Pang observed in \cite{Pang-93} that the behavior of $\zeta_1$ is essentially exact. Secondly, the term $\euler^{-\Lambda (t-t\wedge R^2)}$ clearly  governs the sharp large-time behavior, cf.~\cite{KellerLVW-15}. 
The exponential decay starts after time $t=R^2$ when the volume terms become constant, cf.~\cite[Theorem~15.14]{Grigoryan-09}. In the case $R=\infty$ we have $\Lambda=0$ due to polynomial volume growth implied by volume doubling, \cite{HaeselerKW-13}. Finally, the polynomial correction term in the denominator is discussed below.
\\
Our estimates refine Davies' a priori bound \cite{Davies-93} and the Gaussian upper bounds obtained in \cite{BarlowChen-16} for the normalized Laplacian. Moreover, it refines the sharp Davies-Gaffney-Grigor'yan Lemma of \cite{BauerHuaYau-15,BauerHuaYau-17} even for unbounded Laplacians as can be seen below.

Next, we discuss the correction terms. Rather surprising is the polynomial correction term in time and space: it is easy to see that  for $  r\geq 0, t>0 $
\[
\left(1\vee\sqrt{t^2+r^2}-t\right)^{\frac{n}{2}}\leq \left(1\vee\frac{r^2}{t}\right)^{\frac{n}{2}}.
\]
The right-hand side of this inequality is the polynomial correction term appearing in similar estimates on manifolds, cf.~\cite{Sturm-95,Grigoryan-09}. 
Here, we obtain a better correction term on graphs which satisfies 
\[
\left(1\vee\sqrt{t^2+r^2}-t\right)^{\frac{n}{2}}\to 1\vee r^{\frac{n}2}, \quad t\to 0, 
\]
and 
\[ \left(1\vee\sqrt{t^2+r^2}-t\right)^{\frac{n}{2}}\to 1,\quad t\to \infty.
\]
As mentioned above, an upper bound for the heat kernel of the normalized Laplacian with unbounded combinatorial geometry for large times has been obtained by Barlow and Chen, \cite{BarlowChen-16}. There, the authors assume the Poincar\'e inequality and volume doubling for all balls with fixed center and radii in some given interval whose lower bound is large enough. Additionally, they assume that the volume of balls is comparable from above in terms of the measure of its center. Our result only uses the Sobolev inequality and volume doubling on large balls instead and gives an essentially optimal bound. Moreover, our estimates hold for all times larger than a certain threshold.

In order to formulate our next results, we introduce the following weighted means of the degree and the inverse measure.
Here,  the weighted vertex degree is given, for $ x\in X $, by
\[
\Deg(x):=\frac{\deg(x)}{m(x)}=\frac{1}{m(x)}\sum_{y\in X} b(x,y).
\]
For $x\in X$, $R\geq 0$, and $p\in(1,\infty)$, we define 
\begin{align*}
D_p(R)&:=D_p(x,R):=\left(\frac{1}{m(B(R))}\sum_{y\in B(R)}m(y)\Deg(y)^p\right)^{\frac{1}{p}},\quad
\\
 M_p(R)&:=M_p(x,R):=\left(\frac{1}{m(B(R))}\sum_{y\in B(R)}m(y)\frac{1}{m(y)^p}\right)^{\frac{1}{p}},
\end{align*}
and 
\[
D_\infty(R):=D_\infty(x,R):=\sup_{B(R)} \ \Deg \quad \text{and}\quad M_\infty(R):=M_\infty(x,R):=\sup_{B(R)}\ \frac{1}{m}.
\]
Averages of edge weights and their inverses have been used in \cite{AndresDS-16} to show upper heat kernel bounds for elliptic operators on regular combinatorial graphs.\\
For given $ p $, denote by $q$ be the H\"older conjugate of $p$
\begin{align*}
	\frac{1}{p}+\frac{1}{q}=1
\end{align*}
and for given parameters $ d>0 $, $ p>1 $ and $ R>0 $, let
\begin{align*}
	\mu_x(R):=1\vee \left(\frac{m(B_x(R))}{R^d}\right)^q.
\end{align*}

The second special case is the one of uniformly positive measure, i.e.,
 we assume $\inf_{X} m>0$. This includes the case where $m$ is the counting measure, i.e., $m(x)=1$, $x\in X$. In this case $ M_p $ is obviously bounded.

\begin{theorem}[Uniformly positive measure]\label{thm:counting} 
Let $\inf_{X} m>0$,
 $d>0$, $n>2$, $p\in(1,\infty]$, and $\beta=1+1/(n\vee 2q)$. There exists $R_0>0$  such that  
	if $R\geq 2 r\geq 2 R_0$ there are constants $C_0,t_0>0$ such that for all $x,y\in X$ satisfying $SV(r,R)$ and $t\geq t_0$ we have
	\begin{align*}
		p_{t}(x,y)
		&\leq 
		\left[\mu_x({r})\mu_y({r})\left(1+\tau^2D_p(x,\tau)\right)\left(1+\tau^2D_p(y,\tau)\right)
		\right]^{\theta(\tau)}
		\\
		&\quad\cdot C_0
		\frac{\left(1\vee S^{-2}\left(\sqrt{t^2+\rho(x,y)^2S^2}-t\right)\right)^{\frac{n}{2}}}
		{\sqrt{m(B_{x}(\sqrt {t}\wedge R))m(B_{y}(\sqrt {t}\wedge R))}}
		\euler^{-\Lambda (t-t\wedge R^2)-\zeta_S\left(\rho(x,y),t\right)},
	\end{align*}
where $\tau=\sqrt{t/8}\wedge R/2$ and  $\theta(r)\asymp e^{-\gamma\sqrt{r}}$  with  $\gamma=\ln \beta/{4S}$, $ r\ge0 $. Here, the symbol $ \asymp $ means we have two sided estimates with positive constants.
\end{theorem}
Theorem~\ref{thm:counting} yields for the first time heat kernel estimates for so called antitrees. These graphs often serve as an example to show disparity between the discrete and the continuum setting \cite{KellerLW-21}. To illustrate our results we shortly discuss these examples here and refer to \cite{KellerRose-22b} for the details and proofs.

\begin{example}[Antitrees] Consider $X$ being an antitree with standard weights, i.e., $ b(x,y)\in\{0,1\} $ and $ m(x)=1 $ for $ x\in X $. 
	Let the size of the combinatorial distance spheres be given by $ s_{k}=\lfloor k^{\gamma}\rfloor $, $ k\in\mathbb{N}_{0} $ with $ \gamma\in(0,2) $ and the integer function $ \lfloor\cdot\rfloor $. An antitree with respect to $ (s_{k})_{k\in\NN_0} $ is a graph that has $ s_{k} $ vertices in the $ k $-th combinatorial distance sphere $ S_{k} $ with respect to a root vertex $ o $, where the subgraphs of $ S_{k} \cup S_{k+1}$ are complete bipartite graphs, $ k\in \mathbb{N}_{0} $. The corresponding Laplacian is an unbounded operator with $ \Lambda=0 $. 
	We consider the intrinsic metric $ \rho $ given in Example~\ref{ex:intrinsic} which has finite distance balls and jump size $ S\leq 1 $. As the heat kernel  is spherically symmetric, one can reduce the consideration to a one-dimensional graph  which  satisfies $ SV(r,\infty) $ with   $ d=2(\gamma+1)/(2-\gamma) $, $ n= 2d> 2 $ and $ D_{p}(x,r) $  is polynomially bounded for $r\ge \rho(x,o)  $. Hence, we obtain for $ x,y $ not in the same sphere and large $ t $ by Theorem~\ref{thm:counting} above 
	\begin{align*}
	p_{t}(x,y)
	\leq 
		 C_0
	\frac{\left(1\vee \left(\sqrt{t^2+\rho(x,y)^2}-t\right)\right)^{\frac{n}{2}}}
	{\sqrt{m(B_{x}(\sqrt {t}))m(B_{y}(\sqrt {t}))}}
	\euler^{-\zeta_1\left(\rho(x,y),t\right)}.
\end{align*}	
Furthermore, observe that $ m(B_{z}(r))=\# B_{z}(r) \asymp  r^d $ for $ r\ge \rho(z,o) $.
\end{example}

The final special case we highlight here shows a version of Theorem~\ref{thm:main1sturm} for vertices in a graph with at most superexponentially growing $D_p$ and $M_p$.
\begin{theorem}[Degenerating measure]\label{thm:infty}
Let $d,R_1>0$, $n>2$, $p\in(1,\infty]$, and set 
$ \beta=1+{1}/({n\vee 2q})$ and $\gamma=\ln \beta/{4S}$.  Let $Y\subset X$ and assume there exists $C>0$ that for all $x\in Y$  we have $SV(R_1,\infty)$ and for $  R\geq R_1 $
\[
M_p(x,R), D_p(x,R)\leq C\ \exp({\exp(\gamma\sqrt{r})}).
\]
There exist $C_0,t_0>0$  such that for all $x,y\in Y$ and $t\geq t_0$ we have  
\begin{align*}
p_{t}(x,y)
&
\leq C_0\mu_1^{\theta(\sqrt{t/8})}
\frac{\left(1\vee S^{-2}\left(\sqrt{r^2S^2+t^2}-t\right)\right)^{\frac{n}{2}}}
 {\sqrt{m(B_x(\sqrt {t})) m(B_y(\sqrt {t}))}}
\euler^{-\zeta_S(\rho(x,y),t)},
\end{align*}
where $\mu_1:=\sup_{z\in Y}\mu_z(R_{1})$ and  $\theta(r)\asymp e^{-\gamma\sqrt{r}}$ for $ r\ge0 $.
\end{theorem}

Note that if we assume $Y=X$ in Theorem~\ref{thm:infty} above, then this implies that $M_p$ and $D_p$ must be bounded.

\begin{remark} Observe that the constants $ C_0 $ and $ t_{0} $  in the theorems above can be calculated explicitly from the constants in Theorem~\ref{thm:main1sturm}, Lemma~\ref{lemma:gammanormalized}, and Remark~\ref{remark:gamma}. 
\end{remark}

Anchored upper bounds for heat kernels of elliptic operators with possibly unbounded edge weights have been obtained in \cite{AndresDS-16}.
However, their hypotheses are imposed on the underlying combinatorial graph. 
The authors assume a weaker Sobolev inequality and Ahlfors regularity anchored at a fixed vertex for all large balls. Moreover, they impose boundedness of means of sums of the edge weights and their inverses. Their main results cover the normalizing and the counting measure. For the counting measure, they assume  bounded combinatorial vertex degree, and the estimates involve a so-called \emph{chemical distance}, which is intrinsic in this special case. \\
In a subsequent work \cite{KellerRose-22b} we will show anchored Gaussian bounds under the hypotheses of the present paper.
Furthermore, we will adress  characterizations of upper and lower bounds for heat kernels, i.e., Harnack principle, and Poincar\'e inequality and volume doubling later on.

\subsection{Strategy and discussion of the proof}

Our proof relies on a variant of Davies' method, \cite{Davies-87}, which is a nowadays well-established strategy to prove upper heat kernel bounds. We combine this with a new $\ell^2$-mean value inequality, (MV$ _{2} $) for short, for unbounded graph Laplacians. 
\\

Davies \cite{Davies-87} showed that off-diagonal upper  bounds are a consequence of on-diagonal upper bounds. To this end, for the semigroup $ P $ of an elliptic operator, he considers the sandwiched semigroup $ P^{\omega}=e^{-\omega}Pe^{\omega} $ for certain test functions $ \omega. $ As on-diagonal heat kernel bounds are equivalent to $ \log $-Sobolev inequalities, it suffices to derive such inequalities for the sandwiched operator from the given one. This yields pointwise on-diagonal bounds for the kernel  $P^\omega$ and optimizing over $\omega$ leads to the desired off-diagonal bounds. The approach reduces the problem to prove on-diagonal upper bounds for the sandwiched heat kernel which was pursued in a multitude of articles, including \cite{Davies-93,Davies-93a,Pang-93, CarlenKS-87,Zhikov-13,AndresDS-16}.
\\
Our approach is based on a somewhat similar observation. We show that (MV$ _{2} $) for $P^\omega$ on a space-time cylinder $I\times B$ implies upper heat kernel bounds in terms of operator norms of $P^\omega$ for bounded $\omega$. The integrated maximum principle for Lipschitz functions  obtained in \cite[Lemma~3.3]{BauerHuaYau-17} then gives control on the norms of the semigroups. From there on we follow Davies' idea from \cite{Davies-93} and optimize the resulting bound with respect to the Lipschitz constants. This yields an upper heat kernel bound containing the graph-specific Gaussian term $\euler^{-\zeta_S(\rho,t)}$ and additional terms depending on the geometry of $I\times B$. This is then implemented in Section~\ref{section:davies}.
\\

Next, we derive (MV$ _{2} $) for the sandwiched operator depending on the geomety of the graph. The 
different aspects of the unbounded geometry pose the main challenges in the proof of  (MV$ _{2} $).
First, in contrast to \cite{Delmotte-99}, the combinatorial geometry is not assumed to be bounded, and the Sobolev inequality and volume doubling only hold for radii larger than a given threshold. Secondly, due to the unbounded  weighted vertex degree and inverse measure, the Laplacian is  a priori neither bounded nor uniformly elliptic. Hence, we cannot apply results from \cite{CarlenKS-87, BarlowChen-16,MourratOtto-16}. Moreover, we do not assume that the functions $D_p$ and $M_p$ are bounded. 
\\
In order to prove the (MV$ _{2} $) we adapt the classical Moser iteration scheme, \cite{Moser-64}, which was already applied to graphs in, e.g., \cite{Delmotte-99} or \cite{AndresDS-16}. That is, we iterate $\ell^p$-norms of (sub/super-) solutions of the sandwiched heat equation with $p\to\infty$ on balls $B(R)$ with radius shrinking from radius $R$ to radius $R/2$. 
\\
Contrary to the continuum, the iteration procedure cannot be carried out indefinitely: if one iterates the radii, say along the sequence $R_k=(1+2^{-k})R/2$, $k\to\infty$, it is not clear that the number of vertices contained in $B(R_k)$ decreases in every step.  In \cite{Delmotte-99} and \cite{AndresDS-16}, this was resolved by case distinction. Hence, the geometric restrictions on the graph have to be assumed a priori on the whole domain of iteration. This gives no control on the first appearance of a geometric influence of the small-scale graph structure on the (MV$ _{2} $).
\\
We resolve this problem by splitting the proof of  (MV$ _{2} $) into two parts: First, we iterate subsolutions in space-time cylinders as long as balls for a specific sequence of radii shrink in each step. This yields a bound on the $\ell^p$-norm of subsolutions on the smallest ball attained by this procedure, where $p$ depends on the maximal number of iteration steps. Second, starting from this particular $p$, we iterate in time to the $\ell^\infty$-norm of the solution while keeping the space fixed.
\\
For the first part of the proof, we derive maximal inequalities in distance balls for subsolutions. The key are suitable cut-off functions whose existence is guaranteed by the properties of the intrinsic metric. The maximal inequalities together with the Sobolev inequality yield the initial iteration step, where the radius of the ball shrinks by jump size. Then we iterate this inequality in the interval $[R/2,R]$, giving the maximal value for $p$. This can be found in Section~\ref{section:tas}.
\\
In the second part, we iterate supersolutions instead of subsolutions in time on fixed space. To this end, we   prove maximal inequalities for supersolutions, where the  constants now depend on $D_p$ and $M_p$. We use these maximal inequalities to obtain an $\ell^{\kappa}$-mean value inequality via iteration, where $\kappa$ depends on $p$, independent of $SV(R_1,R_2)$. The proof of this procedure can be found in Section~\ref{section:tfs}.
\\
Combining  the first and second part in Section~\ref{section:l2} yields the desired (MV$ _{2} $). 
\\

Finally, in Section~\ref{section:main} we obtain our main result Theorem~\ref{thm:main1sturm} and derive the heat kernel bound from the (MV$ _{2} $) for the sandwiched operators via our abstract theorem obtained in Section~\ref{section:davies}. 
The proofs of Theorems~\ref{thm:norm}, \ref{thm:counting}, and \ref{thm:infty} can be found there as well.
%
%
%
%
%
%
%
%
%

\section{Moser iteration in time and space}\label{section:tas}

For $I\subset\RR$ and $A\subset X$ a function $u\colon I\times A\to\RR$ is called a (resp. sub-, super-) solution on $I\times A$ if  for any $x\in A$ the map $t\mapsto u_t(x):=u(t,x)$ is continuously differentiable in the interior of $I$ such that the differential has a continuous extension to the closure of $I$ 
and satisfies
\[
(\Delta +\partial_t)u =0\quad\text{on}\ I\times A\quad \mbox{(resp.~$ \leq,\,\geq$)}.
\]
%

Recall the combinatorial interior of a set $ A\subseteq X $
\begin{align*}
	A^{\circ}=\{x\in A\colon b(x,y)=0 \mbox{ for all }y\in X\setminus A \}.
\end{align*}
Furthermore, let $ S $ be the jump size of a given intrinsic metric with distance balls $ B(R) $ about an arbitrary vertex $ x $ which we often suppress in notation here. 

\begin{lemma}\label{balllemma} 
	We have $ B(R-S)\subset B(R)^{\circ}  $ for all $ R \ge0 $.
\end{lemma}
\begin{proof} This follows directly from the definition of $ S $ with the help of the triangle inequality.\end{proof}

\subsection{Maximal inequalities for subsolutions}

For a $\rho$-Lipschitz function $\omega$ we define
\[
\Delta_\omega v(x):=\euler^{\omega(x)}\Delta(\euler^{-\omega}v)(x)\quad\text{and}\quad h(\omega):=\sup_{x\in X} \sum_{y\in X} \frac{b(x,y)}{m(x)}\vert \nabla_{xy}\euler^{\omega}\nabla_{xy}\euler^{-\omega}\vert.
\]
In this paragraph, $v\geq 0$ denotes a non-negative subsolution of $\partial_t+\Delta_\omega$ on $[0,\infty)\times X$, i.e.,
\[
\frac{d}{dt}v+\Delta_\omega v\leq 0.
\] 
The following lemma is a variant of \cite[Lemma~2.2]{AndresDS-16}. Since we deal with different norms and operators, we give a complete proof for reader's convenience.
\begin{lemma}\label{lemma:plusminusv}
Let $I\subset\RR$ be a closed non-empty interval, $B\subset X$ finite, $\phi\colon X\to\RR$ with $0\leq \phi\leq \mathbf{1}_{B\setminus\partial_i B}$, $p\geq 1$. 
Then we have 
\begin{align*}
\frac{d}{d t} \Vert \phi v_t^{p}\Vert_2^2 +\frac{1}{2} \Vert \phi \vert\nabla v_t^p\vert\Vert_2^2
\leq 166p^2\left( h(\omega)+\Vert \vert\nabla\phi\vert\Vert_\infty^2\right)
\Vert \Eins_B v_t^p\Vert_2^2.
\end{align*}
\end{lemma}

To simplify notation in the following we write for $ f\in \ell^{1}(X,m) $
\begin{align*}
	\sum_{X} m f=\sum_{x\in X}m(x)f(x).
\end{align*}

\begin{proof} 
We write $\nabla:=\nabla_{xy}$ and $\psi:=\euler^{\omega}$. (Recall that $(\nabla f)^2\in\cC(X\times X)$ given by  $(\nabla f)^2(x,y)=(f(x)-f(y))$, which is not to be confused with  $\vert \nabla f\vert^2\in\cC(X)$ with $\vert\nabla f\vert^2(x)=\sum_{y\in X}b(x,y)(\nabla f)^2(x,y)$.) 
Use the subsolution property and non-negativity of $v$, the definition of $\Delta_\omega$, and Green's formula to get
\begin{align*}
-\frac{1}{2p}\frac{d}{d t}\sum_{X}m\phi^2v_t^{2p}
&\geq \sum_{X}m\phi^2v_t^{2p-1} \Delta_\omega v_t
=\sum_{X}m\phi^2v_t^{2p-1} \psi\Delta(\psi^{-1} v_t)
\\
&
=\frac{1}{2}\sum_{x,y\in X}b(x,y)(\nabla\phi^2v_t^{2p-1}\psi)(\nabla\psi^{-1}v_t).
\end{align*}

We conclude statement  by applying the following \emph{pointwise} estimate to the right-hand side of the inequality above. Specifically, the term $ \mathbf{1}_{B} $ enters via the observation that  $\nabla\phi $ and $\av(\phi^2)$ are supported on $ B\times B $. 
\begin{claim}\label{claim1}For any $x,y\in B$, we have
\begin{align*}
\nabla(\phi^2v_t^{2p-1}\psi)\nabla(\psi^{-1} v_t)
\!\geq\! 
\frac{1}{2p}\av(\phi^2)(\nabla v_t^p)^2\!-\!(6+160p)\big(\vert\nabla\psi\nabla\psi^{-1}\vert\av(\phi^2)\!+\!(\nabla\phi)^2\big)\av(v_t^{2p}),
\end{align*}
where  $\av(f):=\frac12(f(x)+f(y))$.
\end{claim}
\noindent
\emph{Proof of the claim.}
We have $\nabla(fg)=\av(f)\nabla g+\av(g)\nabla f$. Thus,
\begin{align*}
&(\nabla\phi^2v_t^{2p-1}\psi)(\nabla\psi^{-1}v_t)
= \av(\phi^2) T_1(x,y)+(\nabla\phi^2)T_2(x,y),
\end{align*}
where 
\[
T_1(x,y):=(\nabla v_t^{2p-1}\psi)(\nabla\psi^{-1}v_t),\quad
T_2(x,y):=\av(v_t^{2p-1}\psi)(\nabla\psi^{-1}v_t).
\]
We bound $T_1$ and $T_2$ from below and start with $T_1$. 
Expand using the product rule and apply $\av(\psi)\av(\psi^{-1})=1-\frac{1}{4}\nabla\psi\nabla\psi^{-1}$ and $\av(\psi)(\nabla\psi^{-1})=-\av(\psi^{-1})(\nabla\psi)$ to get 
\begin{align*}
T_1(x,y)
&
=
\av(\psi)\av(\psi^{-1})\nabla v_t^{2p-1} \nabla v_t
+ \av(v_t)\av(v_t^{2p-1})\nabla\psi \nabla\psi^{-1}
\\
&\quad +\av(\psi^{-1}) \av(v_t^{2p-1})\nabla\psi  \nabla v_t
+\av(\psi)\av(v_t)\nabla v_t^{2p-1}\nabla\psi^{-1}
\\
&
=
\nabla v_t^{2p-1} \nabla v_t
+ \bigg(\av(v_t)\av(v_t^{2p-1})-\frac{1}{4}\nabla v_t^{2p-1} \nabla v_t\bigg)\nabla\psi \nabla\psi^{-1}
\\
&\quad + \av(\psi)\nabla\psi^{-1}\bigg(\av(v_t)\nabla v_t^{2p-1} - \av(v_t^{2p-1}) \nabla v_t\bigg) =:S_1+S_2+S_3.
\end{align*}
We bound $S_1,S_2,S_3$ from below and start with $S_1$, where we just apply the elementary inequality (cf.~\cite[(B.1)]{AndresDS-16})
\[
(a^{2p-1}-b^{2p-1})(a-b)\geq \frac{2p-1}{p^2}(a^p-b^p)^2,\quad ,a,b\geq 0, \ p>1/2
\]
to obtain \[S_1\geq \frac{2p-1}{p^2}(\nabla v^p)^2.
\]
Now, we estimate $S_2$. Bound it by its negative modulus and apply the inequalities $\nabla v^{2p-1}\nabla v\leq 2\av(v^{2p})$ and $\av(v^\alpha)\av(v^\beta)\leq \av(v^{\alpha+\beta})$, $\alpha,\beta\in\NN_0$ as a consequence of Young's inequality, to get
\[
S_2\geq -\frac{3}{2}\av(v^{2p})\vert \nabla \psi\nabla \psi^{-1}\vert.
\] 
Last but not least, we estimate $S_3$ from below. First, bound it by its negative modulus. Apply the elementary inequality (cf.~\cite[(B.2)]{AndresDS-16})
\[
\vert a^{2p-1}b-b^{2p-1}a\vert \leq \frac{p-1}{p}\vert a^{2p}-b^{2p}\vert, \quad a,b\geq 0, \ p\geq 1,
\]
to the second factor to obtain 
\begin{align*}
\vert \av(v_t)\nabla v_t^{2p-1}-\av(v_t^{2p-1})\nabla v_t\vert
&=
\vert v_t(x)^{2p-1}v_t(y)-v_t(y)^{2p-1}v_t(x)\vert 
\leq \frac{2(p-1)}{p}\vert\av (v_t^p)\nabla v_t^p\vert.
\end{align*}
Note that $\av(\psi)\av(\psi^{-1})=1+\frac{1}{4}\vert\nabla\psi\nabla\psi^{-1}\vert$ implies
\[
\vert\av(\psi)\nabla\psi^{-1}\vert=\sqrt{\av(\psi)\av(\psi^{-1})}\sqrt{\vert\nabla\psi\nabla
\psi^{-1}\vert}=\sqrt{1+\frac{1}{4}\vert\nabla\psi\nabla\psi^{-1}\vert}\sqrt{\vert\nabla\psi\nabla
\psi^{-1}\vert}.
\]
The above identity and Young's inequality $ab\leq \frac12(a^2/\epsilon+\epsilon b^2)$, $\epsilon>0$ applied to the upper bound on $\vert S_3\vert$ yield
\begin{align*}
\vert S_3\vert
&\leq 
\frac{2(p-1)}{p}\av(\psi)\av(v_t^p)\vert\nabla\psi^{-1} \nabla v_t^p\vert
\\
&\leq
\frac{2(p-1)}{p}\left( 
\frac{1}{\epsilon }\left(1+\frac{1}{4}\vert\nabla\psi\nabla\psi^{-1}\vert\right)(\nabla v_t^p)^2
+
2\epsilon \vert\nabla\psi\nabla\psi^{-1}\vert\av(v_t^p)^2\right)
\\
&
=\frac{1}{\epsilon }\frac{2(p-1)}{p}(\nabla v_t^p)^2
+\frac{2(p-1)}{p}\vert\nabla\psi\nabla\psi^{-1}\vert\left(\frac{1}{4\epsilon}(\nabla v_t^p)^2
+
2\epsilon  \av(v_t^p)^2\right).
\end{align*}
By applying $\av(v^p)^2\leq 4\av(v^{2p})$, $(\nabla v^p)^2\leq 2\av(v^{2p})$ collecting all terms, and choosing $\epsilon=2p$ we obtain 
\[
T_1(x,y)\geq S_1+S_2+S_3
\geq 
\frac{1}{p}(\nabla v_t^{p})^2
-(2+32p)\vert \nabla\psi \nabla\psi^{-1}\vert\av(v_t^{2p}),
\]
since $3/2+2(p-1)/4p^2+32(p-1)\leq 2+32p$. \\
Now we bound $(\nabla \phi^2)T_2$ from below. Expand the factors of $T_2$ by the product rule and bound the resulting summands from below by their negative moduli. Use the inequalities $\av(\psi v_t^{2p-1})\leq 2\av(\psi)\av(v_t^{2p-1})$ and $\av(v^\alpha)\av(v^\beta)\leq \av(v^{\alpha+\beta})$
to obtain
\begin{align*}
(\nabla \phi^2)T_2(x,y)
&=
(\nabla \phi^2)\av(v_t^{2p-1}\psi)(\nabla\psi^{-1}v_t)
\\
&
\geq -2 \vert \nabla \phi^2\av(\psi)\nabla\psi^{-1}\vert\av(v^{2p})-2\av(\psi)\av(\psi^{-1})\vert \nabla \phi^2\av(v^{2p-1}) \nabla v\vert.
\end{align*}
Apply $\nabla\phi^2=2\av(\phi)\nabla\phi$, the elementary inequality (cf.~\cite[(B.3)]{AndresDS-16})
\[
(a^{2p-1}+b^{2p-1})\vert a-b\vert\leq  4\vert a^{p}-b^{p}\vert (a^p+b^p), \quad a,b\geq 0, \ p\geq \frac{1}{2},
\]
and Young's inequality to the modulus involved in the second summand to get 
\begin{align*}
\vert\nabla\phi^2 \av(v^{2p-1}) \nabla v\vert
&\leq 8\av(\phi) \av(v^p) \vert\nabla\phi\nabla v^p\vert 
\leq \frac{4}{\delta} \av(\phi)^2(\nabla v^p)^2
+4\delta(\nabla\phi)^2\av(v^p)^2.
\end{align*}
Apply $\nabla\phi^2=2\av(\phi)\nabla\phi$, 
 $\vert\av(\psi)\nabla \psi^{-1}\vert=\sqrt{-\nabla\psi\nabla\psi^{-1}}\sqrt{\av(\psi)\av(\psi^{-1})}$,
and Young's inequality to the first summand to get 
\begin{align*}
\vert\nabla\phi^2 \av(\psi)\nabla\psi^{-1}\vert 
&=
2\vert\av(\phi)\nabla\phi\vert \sqrt{-\nabla\psi\nabla\psi^{-1}}\sqrt{\av(\psi)\av(\psi^{-1})}
\\
&\leq \av(\phi)^2\vert\nabla\psi\nabla\psi^{-1}\vert+(\nabla\phi)^2\av(\psi)\av(\psi^{-1}).
\end{align*}
Plugging in the above estimates, rearranging, and  using $\av(v^p)^2\leq \av(v^{2p})$, as well as $\av(\psi)\av(\psi^{-1})=1-\tfrac{1}{4}\nabla\psi\nabla\psi^{-1}$ and $(\nabla v^p)^2\leq 2\av(v^{2p})$ yield
\begin{align*}
&(\nabla \phi^2)T_2(x,y)
\geq  
-\frac{8}{\delta} \av(\phi)^2(\nabla v^p)^2
\\
&\quad
-\left(
\left(2 
+8\delta \right)(\nabla\phi)^2+\vert\nabla\psi\nabla\psi^{-1}\vert\left(
\left(2+\frac{4}{\delta}\right) \av(\phi)^2+\left(\frac{1}{2}
+2\delta \right)(\nabla\phi)^2
\right)
\right)
\av(v^{2p}).
\end{align*}

Hence, collecting all the terms, using $(\nabla\phi)^2\leq 2\av(\phi)^2$, $\av(\phi^2)\leq 2\av(\phi)^2\leq 2\av(\phi^2)$,
and choosing $\delta=32p$ leads to the claim.
\end{proof}

\begin{lemma}\label{lemma:maxandintgeneral}Let $T_1\leq T_2$, $B\subset X$ finite, $\phi\colon X\to[0,\infty)$, $ 0\leq \phi\leq \mathbf{1}_{B^{\circ}}  $,
 $\chi\colon [T_1,T_2]\to \RR$ piecewise differentiable, 
\[\eta\colon [T_1,T_2]\times B\to\RR,\qquad (t,x)\mapsto\eta_t(x):=\chi(t)\phi(x),
\] 
and $p\geq 1$. Then, we have
\begin{multline*}
\Bigg[\|\eta_tv_t^{p}\|_{2}^{2}\Bigg]_{t=T_1}^{T_2}+\frac{1}{2}\int_{T_1}^{T_2}\left\|\eta_t\vert\nabla v_t^{p}\vert\right\|^2_{2}\drm t 
\\
\leq \int_{T_1}^{T_2} \sum_{ B}mv_t^{2p}
\cdot\!
\left(2\eta_t\frac{d}{d t}\eta_t\!+\!166p^2(h(\omega)\chi^2(t)+\|\vert\nabla \eta_t\vert\|_{\infty}^{2})\right)\!\drm t.
\end{multline*}
\end{lemma}
\begin{proof}We compute and apply Lemma~\ref{lemma:plusminusv}  to obtain 
\begin{align*}
\frac{d}{dt}\sum_{B} m\eta_t^2v_t^{2p}
&=2\sum_{B} mv_t^{2p}\eta_t\frac{d}{dt}\eta_t
+\chi(t)^{2}\frac{d}{dt}\sum_{ B} m\phi^2v_t^{2p}
\\
&
\leq 2\sum_{B} mv_t^{2p}\eta_t\frac{d}{dt}\eta_t-\frac{1}{2} \Vert \eta_{t}\vert \nabla v\vert\Vert_2^2
\\
&\quad+166p^2\left( h(\omega)\chi(t)^2+\Vert\vert\nabla\eta_t\vert\Vert_\infty^2\right)
\sum_{B}mv_t^{2p}.
\end{align*}
Integration with respect to $t$ yields the result.
\end{proof}

The key for the next lemma is the existence of good cut-off functions with respect to intrinsic metrics. For such a metric $ \rho $, a subset $ A\subset X $, and a radius $ R\ge0 $ we let  $ B_A(R)=\{x\in X\colon \inf_{a\in A}\rho(x,a)\leq R \} $.
\begin{proposition}[{\cite[Proposition~11.29]{KellerLW-21}}]\label{KLW} Assume that $\rho$ is an intrinsic metric, $A\subset X$, $R\geq 0$, and
\[
\phi_{A,R}=\left(1-\frac{\rho(\cdot,A)}{R}\right)_+.
\]
Then, we have $1_{A}\leq \phi_{A,R} \leq 1_{B_{A}(R)}$ and 
\[
\Vert \vert \nabla\phi_{A,R}\vert\Vert_\infty^2\leq \frac1{R^2}.
\]
\end{proposition}

With the help of such cut-off functions we get the following lemma for an intrinsic metric with finite jump size $ S $. 

\begin{lemma}\label{lemma:maxandint} Let $R_1,R_2\geq 0$, $R_2-S>R_1$, $T_1<T_2<T_3$, 
and $p\geq 1$. Then we have
\begin{multline*}
\max_{t\in [T_2,T_3]}\Vert \Eins_{B(R_1)}v_t^{p}\Vert_2^2+
\int_{T_2}^{T_3}\Vert\Eins_{B(R_1)} \vert\nabla v_t^{p}\vert\Vert_2^2\ \drm t \nonumber
\\
\leq 332p^2\left(h(\omega)+\frac{1}{(R_2-R_1-S)^2}+\frac{1}{T_2-T_1}\right)
\int_{T_1}^{T_3} \Vert\Eins_{B(R_2)}v_t^{p}\Vert_2^2\drm t.
\end{multline*}
\end{lemma}

\begin{proof}
Proposition~\ref{KLW} together with Lemma~\ref{balllemma}  implies for 
the choice of cut-off function $\phi:=\phi_{B(R_1),R_2-R_1-S}$ that  $1_{B(R_{1})}\le \phi \le 1_{B(R_{2})^{\circ}}$.
Choose
\[
\chi(t):=
\begin{cases}
\frac{t}{T_2-T_1}-\frac{T_1}{T_2-T_1}
&\colon t\in[T_1,T_2],\\ 
1 
& \colon t\in[T_2,T_3].\end{cases}
\]
We infer from $\chi\leq 1$ and Proposition~\ref{KLW} that $\eta:=\chi\phi$ satisfies for any $t\in[T_1,T_3]$
\[
\Vert\vert \nabla \eta_t\vert\Vert_\infty^2\leq \frac{1}{(R_2-R_1-S)^2}.
\] 
Since $\chi(T_1)=0$, $\phi\leq 1$, and $$ \eta\tfrac{d}{dt}\eta=\chi\phi^{2}\tfrac{d}{dt}\chi\leq \tfrac{1}{T_2-T_1}, $$ 
Lemma~\ref{lemma:maxandintgeneral} implies for any $\tau\in[T_1,T_3]$
\begin{align*}
&\Vert\Eins_{B(R_1)}v_{\tau}^{p}\Vert_2^2+\frac{1}{2}\int_{T_1}^{\tau}\Vert\chi(t)\Eins_{B(R_1)}\vert\nabla v_t^{p}\vert\Vert_2^2\drm t 
\\
&\leq \Bigg[\Vert\eta_{t}v_{t}^{p}\Vert_2^2\Bigg]_{t=T_{1}}^{^\tau}+\frac{1}{2}\int_{T_1}^{\tau}\Vert\eta_{t}\vert\nabla v_t^{p}\vert\Vert_2^2\drm t
\\
&\leq \int_{T_1}^{T_2} \sum_{ B(R_2)}mv_t^{2p}
\cdot\!
\left(2\eta_t\frac{d}{d t}\eta_t\!+\!166p^2(h(\omega)\chi(t)^2+\|\vert\nabla \eta_t\vert\|_{\infty}^{2})\right)\!\drm t\\
&\leq \left(\frac{2}{T_2-T_1}+166p^2\left(h(\omega)+\frac{1}{(R_2-R_1-S)^2}\right)\right)
\int_{T_1}^{\tau} \Vert\Eins_{B(R_2)}v_t^{p}\Vert_2^2\drm t.
\end{align*}
By continuity there exists $\tau\in[T_2, T_3]$ such that 
\begin{align*}
\Vert\Eins_{B(R_1)}v_\tau^{p}\Vert_2^2=\max_{t\in [T_2,T_3]}\Vert\Eins_{ B(R_1)} v_t^{p}\Vert_2^2.
\end{align*}
The above inequality yields an upper bound for this maximum by dropping the second (non-negative) integral. 
Choosing $\tau=T_3$ yields an upper bound for the second summand on the left-hand side by dropping the first summand. The claim follows by noting that $\chi=1$ on $[T_2,T_3]$,  adding the two resulting inequalities, and using $4\leq 332p^2$.
\end{proof}

%
%
%
%
%
%
%
%
%
\subsection{The iteration steps}
%
%
%
%
%
%
%
%
%

%
%
As above we suppress the center $ x  $ in the notation of distance balls $ B(R)=B_{x}(R) $ for $ R\ge0 $. We denote for $ A\subset X $
\[
m_A:=\frac{m}{m(A)}.
\]

The following proposition is the basic ingredient for the Moser iteration in space and time. It shows that the Sobolev inequality implies an estimate of averaged $L^p$-norms in space-time. 
\begin{proposition}\label{prop:parabolicsub1}Let $ x\in X $, $n>2$, 
$\alpha=1+\frac{2}{n}$, 
\[
T_1<T_2<T_3,\quad 0\leq R_1<R_2<R_3,\quad R_1<R_2-S,\quad R_2<R_3-S, 
\]
 $p\geq 1$, and $v\geq 0$ a subsolution on  $[T_1,T_3]\times B(R_3)$. 
If  $S(n,R_2)$ holds in $ x $, then 
\begin{align*}
\frac{1}{T_3-T_2}&\int_{T_2}^{T_3} \sum_{B(R_1)}m_{B(R_1)}v_t^{2p\alpha}
\leq C_0
\left(\frac{1}{T_3-T_1}\int_{T_1}^{T_3}\sum_{B(R_3)}m_{B(R_3)}v_t^{2p}\right)^\alpha
\end{align*}
with $C_0=C_0(x,p,R_1,R_2,R_3,T_1,T_2,T_3,n)$ given by
\begin{multline*}
C_0
= C_S996^\alpha p^{2\alpha} R_2^2\frac{m(B(R_3))}{m(B(R_1))}  \left(\frac{m(B(R_3))}{m(B(R_2))}\right)^{\frac{2}{n}}
\\
\cdot\frac{(T_3-T_1)^\alpha}{T_3-T_2}\left(h(\omega)+\frac{1}{(R_3-R_2-S)^2}+\!\frac{1}{(R_2-R_1-S)^2}+\frac{1}{T_2-T_1}\!\right)^\alpha .
\end{multline*}
\end{proposition}

\begin{proof}
Since $\alpha=1+2/n$, we infer from H\"older's inequality for $w\colon [T_2,T_3]\times X\to[0,\infty)$
\begin{align*}
\int_{T_2}^{T_3}\sum_{B(R_1)}mw_t^{2\alpha} \drm t 
&=\int_{T_2}^{T_3}\sum_{B(R_1)}mw_t^2w_t^{\frac{4}{n}}\drm t
\\
&\leq 
\int_{T_2}^{T_3}\left(\sum_{B(R_1)}m w_t^{\frac{2n}{n-2}}\right)^{\frac{n-2}{n}}\left(\sum_{B(R_1)}mw_t^{2}\right)^{\frac{2}{n}}\drm t
\\
&
\leq 
\max_{t\in[T_2,T_3]}\left(\sum_{B(R_1)}mw_t^{2}\right)^{\frac{2}{n}}
\int_{T_2}^{T_3}\left(\sum_{B(R_1)}m w_t^{\frac{2n}{n-2}}\right)^{\frac{n-2}{n}}\drm t.
\end{align*}
We bound the integrand appearing on the right-hand side from above. Choose the cut-off function $\phi=\phi_{B(R_1), R_2-R_1-S}$.
Since $\supp\phi\subset B(R_2-S)\subset B(R_2)^\circ$, the definition of $\phi$, the Sobolev inequality $S(n,R_2)$, and $\nabla_{xy}(fg)=f(x)\nabla_{xy}g+g(y)\nabla_{xy}f$ imply for any $t\in [T_2,T_3]$
\begin{align*}
&\frac{m(B(R_2))^{\frac{2}{n}}}{C_SR_2^2}
\left(\sum_{B(R_1)}m w_t^{\frac{2n}{n-2}}\right)^{\frac{n-2}{n}}
\leq\frac{m(B(R_2))^{\frac{2}{n}}}{C_SR_2^2}
\left(\sum_{B(R_2)} m(\phi w_t)^\frac{2n}{n-2}\right)^{\frac{n-2}{n}}
\\
&\leq
\Vert\Eins_{B(R_2)}\vert\nabla(\phi w_t)\vert\Vert_2^2+\frac{1}{R_2^2}\Vert\Eins_{B(R_2)} \phi w_t\Vert_2^2
\\
&\leq
2\sum_{B(R_2)} m\vert\nabla\phi\vert^2w_t^2+2\!\!\sum_{x\in B(R_2)}m\vert\nabla w_t\vert^2\phi^2
+\frac{1}{R_2^2}\Vert\Eins_{B(R_2)} w_t\Vert_2^2\\
&\leq 
2\Vert\Eins_{B(R_2)}\vert\nabla w_t\vert\Vert_2^2
+\left(\frac{2}{(R_2-R_1-S)^2}+\frac{1}{R_2^2}\right)\Vert\Eins_{ B(R_2)}
 w_t\Vert_2^2
 \\
&\leq 
3\left[\Vert\Eins_{B(R_2)}\vert\nabla w_t\vert\Vert_2^2
+\frac{1}{(R_2-R_1-S)^2}\Vert\Eins_{ B(R_2)}
 w_t\Vert_2^2\right]
.
\end{align*}
where we use  Proposition~\ref{KLW} applied to $ \phi $ in the fourth inequality.
Hence, we get 
\begin{multline*}
\frac{m(B(R_2))^{\frac{2}{n}}}{C_SR_2^2}\frac{1}{T_3-T_2}\int_{T_2}^{T_3} \sum_{B(R_1)}m_{B(R_1)}w_t^{2\alpha}\ \drm t
\\
\leq 3\cdot 
\frac{1}{m(B(R_1))}\frac{1}{T_3-T_2}
\max_{t\in[T_2,T_3]}\left\Vert\Eins_{B(R_1)}w_t\right\Vert_{2}^{\frac{2}{n}} 
\\\cdot
\int_{T_2}^{T_3}\left( \Vert\Eins_{B(R_2)}\vert\nabla w_t\vert\Vert_2^2
+\frac{1}{(R_2-R_1-S)^2}\Vert\Eins_{ B(R_2)}
 w_t\Vert_2^2\right)\drm t 
.
\end{multline*}
Now, let $w=v^p$. We observe
 $ \Vert\Eins_{B(R_1)}v_t^{p}\Vert_{2}\leq \Vert\Eins_{B(R_2)}v_t^{p}\Vert_{2} $ and we apply  Lemma~\ref{lemma:maxandint} to this term as well as to $ \Vert\Eins_{B(R_2)}\vert\nabla v_t^{p}\vert\Vert_2^2 $ with
 the choice of radii $R_3-S>R_2$. Collecting the remaining factors into the constant $ C_{0} $ yields the statement. 
\end{proof}

The theorem below shows that the Moser iteration procedure delivers a weaker result than in the continuum setting. More precisely, instead of a bound for the $\ell^\infty$-norm of a subsolution in terms of its $\ell^2$-norm, it allows only to bound a certain $\ell^p$-norm of subsolutions. The value $p$ depends on the number of possible iteration steps in space.

\begin{theorem}[Moser iteration in time and space]\label{thm:subsolutionsogeneral}
Let $ x\in X $, $T$, $d>0$, 
$n>2$, $\alpha=1+2/n$, $\delta\in(0,1]$, $R\geq 32S$, and 
\begin{align*}
K=K(R,S)= 
\left\lfloor\sqrt{\frac{R}{8S}}-2\right\rfloor.
\end{align*}
Assume $SV(n,R/2,R)$ in $ x$.
 For all non-negative $\Delta_\omega$-subsolutions $v\geq 0$ on the cylinder $[T-R^2,T+R^2]\times B(R)$ we have
 \begin{multline*}
\left(\frac{1}{2  \delta (R/2)^2}\int\limits_{ T-\delta (R/2)^2}^{T+\delta (R/2)^2}\sum_{B(R/2)} m_{B( R/2)}v_t^{2\alpha^{K}}\drm t\right)^{\alpha^{-K}}
\\
\leq 
 \frac{C_{d,n}(1+\delta R^2h(\omega))^{\frac{n}{2}+1}  }{\delta^{\frac{n}{2}+1} R^{2}}
\int\limits_{T- \delta R^2}^{T+\delta R^2}\sum_{B(R)} m_{B(R)}v_t^{2}\drm t,
\end{multline*}
where $C_{d,n}:=(1\vee C_D)\left(1\vee C_D^{\frac{n}{2}+1}C_S^{\frac{n}{2}}\right)10^{8((n+2)(d+1)+n^2+n)+1}$.
\end{theorem}
\begin{proof}By assumption on $R$ we have $K\geq 0$.
Set 
\[
 \rho_k:= \frac{R}{2}\left(1+\frac{1}{k+1}\right), \quad k\ge0.
\]
We have $ \rho_{0}=R $, $
\rho_{k}\geq  R/2
$ for $k\in\NN_0$ and,  since $ K\le \sqrt{R/8S}-2 $,  we obtain
\[ 
\rho_{k+1}+4S\leq \rho_k,\quad k\in\{0,\ldots,K\}.
\]
Together with Lemma~\ref{balllemma} this yields
\[
B(\rho_{k+1})\subset 
B(\rho_k-S) \subset B(\rho_k)^{\circ}, \quad k\in\{0,\ldots, K\}.
\]
Hence, we can iterate $K$ times. 
We apply Proposition~\ref{prop:parabolicsub1} to $v$ and $p=\alpha^k$, where $k\in\{0,\ldots,K\}$,
\[
R_1=\rho_{k+1},\quad R_2=(\rho_{k}+\rho_{k+1})/2, \quad R_3=\rho_k,
\]
and with $ \delta\in (0,1] $
\[
T_1:=T-\delta\rho_k^2, \quad T_2:=T-\delta\rho_{k+1}^2, \quad
T_3:=T+\delta\rho_k^2.
\]
This yields for any $k\in\{0,\ldots,K\}$ since $\rho_k\geq \rho_{k+1}\geq \rho_k/2$ (and using the resulting bound $ 2\rho_{k+1}^{2}\ge 2(\rho_{k+1}^{2}+\rho_{k}^{2})/5 $) and $ T_{2}-T_{1}= \delta(\rho_{k+1}^2+\rho_k^2)$
\begin{multline*}
\frac{1}{2\delta\rho_{k+1}^2}\int\limits_{T-\delta\rho_{k+1}^2}^{T+\delta\rho_{k+1}^2}\sum_{B(\rho_{k+1})} m_{B(\rho_{k+1})}v_t^{2\alpha^{k+1}}\drm t
\\
\leq 
\frac{5}{2}\frac{1}{\delta(\rho_{k+1}^2+\rho_k^2)}\int\limits_{T-\delta\rho_{k+1}^2}^{T+\delta\rho_{k}^2}\sum_{B(\rho_{k+1})} m_{B(\rho_{k+1})}v_t^{2\alpha^{k+1}}\drm t
\\
\leq 
C_{0,k}
\left(\frac{1}{2\delta\rho_{k}^2}\int\limits_{T-\delta\rho_k^2}^{T+\delta\rho_k^2}\sum_{B(\rho_k)} m_{B(\rho_k)}v_t^{2\alpha^k}\drm t\right)^\alpha,
\end{multline*}
where the constant $C_{0,k}$ arising from Proposition~\ref{prop:parabolicsub1} is given by
\[
C_{0,k}:=\frac{5}{2}
C_0(x,\alpha^k,\rho_{k+1},(\rho_{k}+\rho_{k+1})/2,\rho_k,T-\rho_k^2,T-\rho_{k+1}^2, T+\rho_k^2,n).
\]
Using $\rho_{K}\geq R/2$, 
iterating the above inequality from $k= K-1 $ to $k=0 $ yields
\begin{multline*}
\left(\frac{1}{2\delta(R/2)^2}\int\limits_{T-\delta(R/2)^2}^{T+\delta(R/2)^2}\sum_{B( R/2)} m_{B( R/2)}v_t^{2\alpha^{K}}\drm t\right)^{\frac{1}{\alpha^{K}}}
\\
\leq 
\left(\frac{\rho_K^2}{(R/2)^2}\frac{m(B(\rho_K))}{m(B( R/2))}\right)^{\frac{1}{\alpha^K}}
\left(\frac{1}{2\delta\rho_K^2}\int\limits_{T-\delta\rho_{K}^2}^{T+\delta\rho_{K}^2}\sum_{B(\rho_{K})} m_{B(\rho_K)}v_t^{2\alpha^{K}}\drm t\right)^{\frac{1}{\alpha^{K}}}
\\
\leq C_1\int\limits_{T-\delta R^2}^{T+\delta R^2}\sum_{B(R)} m_{B(R)}v_t^{2p}\drm t,
\end{multline*}
where we have as $ \rho_{0}=R $
\begin{align*}
C_1
&
:=\frac{1}{2\delta R^2}\left(\frac{\rho_K^2}{(R/2)^2}\frac{m(B(\rho_K))}{m(B(R/2))}\right)^{\frac{1}{\alpha^K}}\prod_{k=0}^{K-1}C_{0,k}^{\frac{1}{\alpha^{k+1}}}.
\end{align*}

We bound $C_1$ from above and start with the second factor. The volume doubling property $V(d,R/2,R)$, i.e.,
$ m(B_{r_2})\leq C_D \left({r_2}/{r_1}\right)^d m(B_{r_1}) $ for radii $ r_1\leq r_2 $, such that  $\rho_K\in[ R/2, R]$ yields  
\[
\frac{\rho_K^2}{(R/2)^2}\frac{m(B(\rho_K))}{m(B( R/2))}\leq 
\frac{R^2}{(R/2)^2}\frac{m(B(R))}{m(B( R/2))}
\leq 2^{d+2}(1\vee C_D).
\]
Now, we bound $C_{0,k}$ from above. We have from the definition in Proposition~\ref{prop:parabolicsub1}
\begin{multline*}
C_{0,k}
=
\frac52C_S996^\alpha \alpha^{2k\alpha} \frac{(\rho_{k}+\rho_{k+1})^2}{4}\frac{m(B(\rho_k))}{m(B(\rho_{k+1}))} 
\cdot
\left(\frac{m(B(\rho_k))}{m(B(\frac{\rho_k+\rho_{k+1}}{2})}\right)^{\frac{2}{n}}
\\
 \cdot\frac{(2\delta\rho_k^2)^{\alpha}}{\delta(\rho_k^2+\rho_{k+1}^2)} 
\left(h(\omega)+\frac{4}{(\rho_k-\rho_{k+1}-2S)^2}
+\frac{4}{(\rho_{k}-\rho_{k+1}-2S)^2}+\frac{1}{\delta(\rho_k^2-\rho_{k+1}^2)}\right)^{\alpha}.
\end{multline*}
Since $\rho_k\leq R$, $k\in\{0,\ldots, K\}$, we have 
\[
\frac{(\rho_{k}+\rho_{k+1})^2}{4} \leq R^2.
\]
As above by volume doubling and $\rho_{k+1}\le \rho_k$ taking values in $[R/2, R]$, $k\in\{0,\ldots, K\}$,
\[
\frac{m(B(\rho_k))}{m(B(\frac{\rho_k+\rho_{k+1}}2))}
\leq \frac{m(B(\rho_k))}{m(B(\rho_{k+1}))}\leq \frac{m(B(R))}{m(B( R/2))}\leq 2^dC_D.
\]
Further, $\rho_k\leq R$, $k\in\{0,\ldots,K\}$, and $\alpha>1$ yield 
\begin{align*}
 \frac{(2\delta\rho_k^2)^{\alpha}}{\delta(\rho_k^2+\rho_{k+1}^2)} 
\leq 
\frac{(2\delta\rho_k^2)^{\alpha}}{\delta\rho_k^2} 
&
\leq 
2^{\alpha}(\delta R^2)^{\alpha-1}.
\end{align*}
Now, we estimate the last factor in $ C_{0,k} $. All terms can be bounded from above by an upper  bound on $1/(\rho_k-\rho_{k+1})$. The definition of $\rho_k$ yields
\[
\rho_k-\rho_{k+1}\geq \frac{R}{2(k+2)^2},\quad k\in\NN_0.
\]
By the definition of $K$ we have $ k\leq K\leq \sqrt{{R}/{8S}}-2.
$ for all $k\in\{0,\ldots,K\}$ 
 which  implies 
\[
2S\leq \frac{R}{4(k+2)^2},
\]
such that 
\[
\rho_k-\rho_{k+1}-2S
\geq \frac{R}{2(k+2)^2}-2S
\geq \frac{R}{4(k+2)^2}
.
\]
Moreover, since $\rho_k\geq  R/2$, $k\in\{0,\ldots K\}$,  we have
\begin{align*}
\rho_k^2-\rho_{k+1}^2
=(\rho_k+\rho_{k+1})(\rho_k-\rho_{k+1})
\geq \frac{R^2}{2(k+2)^2}.
\end{align*}
Hence, since $k+2\leq 2^{k+1}$ and $\delta\in(0,1]$, we get 
\begin{align*}
&\left(h(\omega)+\frac{1}{\delta(\rho_k^2-\rho_{k+1}^2)}+\frac{4}{(\rho_k-\rho_{k+1}-2S)^2}
+\frac{4}{(\rho_{k}-\rho_{k+1}-2S)^2}\!\right)^{\alpha}
\\
&
\leq 
\left(
h(\omega)+
\frac{2(k+2)^2}{ \delta R^2 }
+\frac{8\cdot 16(k+2)^4}{R^2}
\right)^\alpha
\\&\leq 
\left(4096\cdot 16^k\right)^{\alpha}\delta^{-\alpha}R^{-2\alpha}
(1+\delta R^2h(\omega))^\alpha.
\end{align*}

Hence,  with $ 5/2\le 3^{\alpha} $ and plugging in the above estimates, we get
\begin{align*}
C_{0,k}
&=\frac52
C_S996^\alpha \alpha^{2k\alpha} \frac{(\rho_{k}+\rho_{k+1})^2}{4}\frac{m(B(\rho_k))}{m(B(\rho_{k+1}))} m(B(\rho_k))^{\frac{2}{n}}m\left(B\left(\frac{\rho_k+\rho_{k+1}}{2}\right)\right)^{-\frac{2}{n}}
\\
&
\quad
\cdot
\frac{(2\delta \rho_k^2)^{\alpha}}{\delta(\rho_k^2+\rho_{k+1}^2)} 
\left(h(\omega)+\frac{4}{(\rho_k-\rho_{k+1}-2S)^2}
+\frac{4}{(\rho_{k}-\rho_{k+1}-2S)^2}+\frac{1}{\delta(\rho_k^2-\rho_{k+1}^2)}\right)^{\alpha}
\\
&\leq 3^\alpha\cdot
C_S\cdot 996^\alpha 2^{d\alpha} C_D^{\alpha}\cdot 2^\alpha\cdot  \alpha^{2k\alpha}\cdot 4096^{\alpha}\cdot 16^{k\alpha}
\delta^{-1}(1+\delta R^2h(\omega))^\alpha
\\
&
\leq (1\vee C_SC_D^\alpha)  10^{8\alpha(1+d)+1}(16\alpha^2)^{\alpha k}
\delta^{-1}(1+\delta R^2h(\omega))^\alpha.
\end{align*}

Finally, this implies using $\rho_K\leq R$ and the estimates on the three factors of $ C_{1} $
\begin{align*}
C_1
&=\frac{1}{2 \delta R^2}\left(\frac{\rho_K^2}{(R/2)^2}\frac{m(B(\rho_K))}{m(B(R/2))}\right)^{\frac{1}{\alpha^K}}\prod_{k=0}^{K-1}C_{0,k}^{\frac{1}{\alpha^{k+1}}}
\\
&
\leq 
(2^{d+2}1\vee C_D)^{\alpha^{-K}}
\left(1\vee C_SC_D^\alpha 10^{8\alpha(1+d)+1}\right)^{\sum_{k=0}^{K-1}\frac{1}{\alpha^{k+1}}}(16\alpha^{2})^{\sum_{k=0}^{K-1}\frac{k}{\alpha^{k}}}
\\
&
\quad\cdot
 \delta^{-1}R^{-2} \delta^{-\sum_{k=0}^{K-1}\frac{1}{\alpha^{k+1}}}
 (1+\delta R^2h(\omega))^{\sum_{k=0}^{K-1}\frac{1}{\alpha^{k}}}.
\end{align*}
Now, use
$\alpha=1+2/n$,
\[
\sum_{k=0}^{K-1}\alpha^{-k}=\frac{\alpha}{\alpha-1}(1-\alpha^{-K})=\left(\frac{n}{2}+1\right)(1-\alpha^{-K}),\quad \frac{\alpha-1}{\alpha}\sum_{j=0}^\infty \frac{j}{\alpha^j}=\frac{1}{\alpha-1},
\]
\[ \sum_{k=0}^{K-1}\alpha^{-(k+1)}=\frac{1}{\alpha}\sum_{k=0}^{K-1}\alpha^{-k}=\frac{1}{\alpha-1}(1-\alpha^{-K})=\frac{n}{2}(1-\alpha^{-K}),
\]
to conclude the statement.
\end{proof}

\begin{remark}One may wonder about the choice for the chosen radii used for the iteration procedure, since usually diadic iteration steps are used. In fact, our choice yields faster convergence of the error function $\Gamma$ introduced in Section~\ref{section:l2}: in our case, the exponent is $r\mapsto \beta^{-\sqrt r}$, in the diadic case it is only $r\mapsto -\beta^{-\log_2 r}$.
\end{remark}
%
%
%
%
%
%
%
%
\section{Moser iteration in time and fixed space}\label{section:tfs}

In this paragraph we let $(B,\mu)$ be a discrete measure space
with finite measure and the norm will be abbreviated by $$ \Vert\cdot\Vert_{p}:=\Vert \cdot\Vert_{\ell^p(B,\mu)}\qquad p\in[1,\infty]. $$ For $p\in[1,\infty]$, we denote by $q$ the H\"older conjugate of $p$, i.e.,
$ \frac{1}{p}+\frac{1}{q}=1 $,
where $q=1$ if $p=\infty$. Below, we will apply these results to $B=B(R)$ and $\mu=m/m(B(R))$.

\subsection{Maximal inequalities for supersolutions on fixed space}
%
%
%
%
%
%
%
%
%

The following lemma can be seen as an interpolation lemma for the embedding operator between discrete $\ell^p$ spaces.

\begin{lemma}\label{lemma:maxnorm}We have for all $v\in \cC(B)$
\[
\Vert v\Vert_{\infty}\leq \Vert \mu^{-1}\Vert_{p}\Vert v\Vert_{q}.
\]
\end{lemma}
\begin{proof}H\"older's inequality yields for all $p\in [1,\infty]$
\[\sup_{ B} |v |
 \leq \sum_{B}\mu\frac{1}{\mu}|v|
\leq \Vert \mu^{-1}\Vert_{p}\Vert v\Vert_{q}.
\hfill\qedhere \]
\end{proof}

The lemma below gives a variant of the maximal inequality for supersolutions in contrast to one for subsolutions in Lemma~\ref{lemma:maxandint}. 
\begin{lemma}\label{lemma:superdiscrete}Let $ p\in [1,\infty] $,
$T_1\leq T_2\leq T_3< T_4$, and $v\geq 0$ a bounded $\Delta_\omega$-supersolution on $[T_1,T_4]\times B$. Then, for any $s\geq 1$, we have 
\begin{equation*}
\max_{[T_2,T_3]}{\Vert v^{s}\Vert}_{1}
\leq 
\left(\frac{\mu(B)^\frac{1}{p}}{T_4-T_3}+s\Vert \Deg\Vert_{p}\right)\int\limits_{T_1}^{T_4}{\Vert v_\tau^s\Vert}_{q}\drm \tau.
\end{equation*}  
\end{lemma}
\begin{proof}
We have for any $t\in (T_1,T_4)$  
\[
\frac{d}{dt} v_t^{s}=
 sv_t^{s-1}\frac{d}{dt} v_t\geq -sv_{t}^{s-1}\Delta_\omega v_t\geq -s\Deg v_t^{s}.
\]
By continuity we can choose $t^\ast\in [T_2,T_3]$ with
\[
\sum_{B}\mu v_{t^\ast}^{s}=\max_{t\in [T_2,T_3]}\sum_{ B}\mu v_{t}^{s}.
\]
Therefore, for any $\tau\in[t^\ast,T_3]$, we have by Fubini's theorem 
\begin{align*}
\sum_{B}\mu  v_\tau^{s}-\sum_{ B}\mu v_{t^\ast}^{s}
&= \sum_{B}\mu  (v_\tau^{s}- v_{t^\ast}^{s})=\sum_{B}\mu \int_{t^\ast}^\tau\frac{d}{dt}v_t^{s}\drm t =
\int_{t^\ast}^\tau\sum_{B}\mu \frac{d}{dt}v_t^{s}\drm t\\
&\geq -s\int_{t^\ast}^\tau \sum_{B} \mu\ v_t^{s}\Deg\ \drm t.
\end{align*}
Integrating with respect to $\tau$ from $t^\ast$ to $T_4$, and setting $L:=T_4-t^\ast$ yields
\begin{align*}
L\sum_{ B}\mu \ v_{t^\ast}^{s}
&\leq \int\limits_{t^\ast}^{T_4}\sum_{B}\mu\ v_\tau^{s}\ \drm \tau
+s\int\limits_{t^\ast}^{T_4}\int\limits_{t^\ast}^\tau \sum_{B} \mu\ v_t^{s}\Deg\ \drm t \ \drm \tau
\\
&\leq 
\int\limits_{T_{2}}^{T_4}\sum_{ B}\mu\ v_\tau^{s}\ \drm \tau
+s L\int\limits_{T_{2}}^{T_4}\!\! \sum_{B} \mu\ v_t^{s}\Deg\  \drm t
\\
&\le  
\left({\mu(B)^\frac{1}{p}}+s{L}\left(\sum_{ B}\mu\Deg^{p}\right)^{\frac{1}{p}}\right)
\int\limits_{T_2}^{T_4}\left(\sum_{ B}\mu\  v_\tau^{sq}\right)^{\frac{1}{q}}\drm \tau,
\end{align*}
where the last step follows from   H\"older's inequality applied to both terms. This yields the claim after dividing by $ L\ge T_{4}-T_{3} $.
\end{proof}

\subsection{The iteration steps}

\begin{lemma}\label{lemma:moserconstant}Let $p\in(1,\infty]$ and $T_1\leq T_2\leq T_3< T_4$, $\beta\in(1,1+\tfrac 1q)$, and $v\geq 0$ a bounded $\Delta_\omega$-supersolution on $[T_1,T_4]\times B$. Then, we have, for all $s\geq 1$,
\begin{align*}
\int\limits_{T_2}^{T_3}\big \Vert v_t^{s\beta }\big\Vert_{q}\drm t
\leq \left[s\left(\frac{1\vee\mu(B)^\frac{1}{p}}{T_4-T_3}+\Vert \Deg\Vert_{p}\right)
\Vert \mu^{-1}\Vert_{p}^q\right]^{\beta-1}
\left(\int\limits_{T_1}^{T_4}\left\Vert v_t^{s}\right\Vert_{q}\drm t\right)^{\beta}.
\end{align*}
\end{lemma}
\begin{proof}
 We have to distinguish between the cases $p\in(1,\infty)$ and $p=\infty$. 
\\
\noindent
\textbf{$\mathbf {1^{\mathrm{st}}}$ case: $p\in(1,\infty)$.}
Let $\gamma:=\tfrac{q}{1-q(\beta-1)}$. We apply H\"older's inequality with exponents $ \tilde p=1/ ((\beta-1)q )$ and $ \tilde q =\gamma/q $, i.e., $ 1/\tilde p +1/\tilde q=1$,
\begin{align*}
\int\limits_{T_2}^{T_3}\big \Vert v_t^{s\beta }\big\Vert_{q}\drm t
&=\int\limits_{T_2}^{T_3}\big \Vert v_t^{s(\beta-1)q }v_{t}^{sq}\big\Vert_{1}^{\frac{1}{q}}\drm t
\\
&\leq \int\limits_{T_2}^{T_3}\big \Vert v_t^{s}\big\Vert_{1}^{(\beta-1)}\big \Vert v_t^{s }\big\Vert_{\gamma}\drm t\nonumber
\leq \max_{[T_2,T_3]}\Vert v_t^{s}\big\Vert_{1}^{(\beta-1)}
\int\limits_{T_1}^{T_4}\Vert v_t^{s }\big\Vert_{\gamma}\drm t.
\end{align*}
We estimate the integrand of the second factor on the right hand side. We apply H\"older's inequality and Lemma~\ref{lemma:maxnorm}, i.e., $ \|v_{t}^{s}\|_{\infty}\leq \|\mu^{-1}\|_{p}\|v_{t}^{s}\|_{q}$,
\begin{align*}
	\Vert v_t^{s }\big\Vert_{\gamma}&=\left(\sum_{ B} \mu\ v_t^{s(\gamma-q)}v_t^{sq}\right)^{\frac{1}{\gamma}} \leq\big\|v_t^{s}\big\|_{\infty}^{(1-\frac{q}\gamma)}
	\|v_{t}^{s}\|_{q}^{\frac{q}{\gamma}}
\le  \Vert \mu^{-1}\Vert_{p}^{(1-\frac{q}{\gamma})}	\|v_{t}^{s}\|_{q}^{(1-\frac{q}{\gamma})+\frac{q}{\gamma}}\\
&=  \Vert \mu^{-1}\Vert_{p}^{q(\beta-1)}	\|v_{t}^{s}\|_{q}
\end{align*}
as $ (1-\frac{q}{\gamma})=q(\beta-1) $. We conclude by estimating the first factor with
 Lemma~\ref{lemma:superdiscrete} 
 \begin{align*}
 \max_{[T_2,T_3]}\Vert v^s\Vert_{1}
\leq 
\left(\frac{\mu(B)^\frac{1}{p}}{T_4-T_3}+s\Vert \Deg\Vert_{p}\right)\int\limits_{T_1}^{T_4}\Vert v_t^s\Vert_{q}\drm t.
 \end{align*} 

\noindent
\textbf{$\mathbf{2^{\mathrm{nd}}}$ case: $p=\infty$.} The proof works for all $\beta>1$, but for simplicity we restrict ourselves to $\beta\in(1,2)$. We closely follow the arguments of the first case. 
We expand $\beta=(\beta-1)+1$ and obtain 
\begin{align*}
	\int\limits_{T_2}^{T_3}\big \Vert v_t^{s\beta }\big\Vert_{1}\drm t
	\leq \max_{[T_2,T_3]}\Vert v_t^{s}\big\Vert_{\infty}^{(\beta-1)}
	\int\limits_{T_1}^{T_4}\Vert v_t^{s }\big\Vert_{1}\drm t  
\end{align*}
Now, we obtain the statement as Lemma~\ref{lemma:maxnorm} and Lemma~\ref{lemma:superdiscrete} give
\[	\max_{[T_2,T_3]}{\Vert v_t^{s}\big\Vert}_{\infty}\leq 	\max_{[T_2,T_3]}{\|v_{t}^{s}\|}_{1}\leq \left(\frac{1}{T_4-T_3}+\Vert \Deg\Vert_{\infty}\right)\int\limits_{T_1}^{T_4}{\|v_{t}^{s}\|}_{1}\drm t.
\hfill \qedhere\]
\end{proof}

The Moser iteration in time will be applied after iteration in space and time cannot be continued due to discreteness of the space. In the next section we apply the theorem below with $ B=B(R) $ and $ \mu=m_{B(R)} $ for some $ R\ge0 $.

\begin{theorem}[Moser iteration in time]\label{thm:constantballsdavies}
Fix real numbers $\delta>0$, $T\geq 0$, $p\in(1,\infty]$, $\beta\in(1,1+\tfrac{1}{q})$, $k\in\NN_0$, and $v\geq 0$ a bounded $\Delta_\omega$-supersolution on $[(1-\delta)T,(1+\delta)T]\times B$. 
Then, 
\begin{align*}
\sup_{[(1-\delta/2)T,(1+\delta/2)T]\times B}\!\! v^2 \leq G_B \left(\frac{1}{2\delta T}\int\limits_{(1-\delta)T}^{(1+\delta)T}\sum_{B}
\mu\ v_t^{2\beta^kq}\ \drm t\right)^{\frac{1}{\beta^kq}}\!\!\!,
\end{align*}
where $G=G_{B,\mu}(\delta,R,T,r,\beta)$ is given by
\begin{align*}
G= C_\beta
\left[\left(1\vee\mu(B)^\frac{1}{p}+\delta T\Vert \Deg\Vert_{p}\right)
\Vert \mu^{-1}\Vert_{p}^{q}\right]^{\frac{1}{\beta^{k}}}, 
\end{align*}
and $C_\beta:=4^{((4+1/\ln\beta)+\beta/(\beta-1))/(\beta-1)}$.
\end{theorem}
\begin{proof} 
For  $ k\in\NN_0 $, denote 
\[
\rho_k:=\left(1+2^{-k}\left(\sqrt{2}-1\right)\right)\sqrt{\frac{\delta T}{2}}.
\]
Then, we have 
$$
\rho_0=\sqrt{\delta T},  \qquad \rho_k\to \sqrt{\frac{\delta T}{2}}, 
\quad k\to\infty.$$
Lemma~\ref{lemma:moserconstant} yields for any $k\in\NN_0$ with 
\[
T_1=T-\rho_k^2,\quad T_2=T-\rho_{k+1}^2,\quad T_3=T+\rho_{k+1}^2,\quad T_1=T+\rho_k^2,
\]
and $ s=2\beta^{k} $ the inequality 
\begin{align*}
\Bigg(\frac{1}{2\rho_{k+1}^2}\int\limits_{T-\rho_{k+1}^2}^{T+\rho_{k+1}^2}\Vert v_t^{2\beta^{k+1}}\Vert_{q}\drm t\Bigg)^\frac{1}{\beta^{k+1}}
\leq 
C_{p,k}^\frac{1}{\beta^{k+1}}
\Bigg(\frac{1}{2\rho_{k}^2}\int\limits_{T-\rho_{k}^2}^{T+\rho_{k}^2}\Vert  v_t^{2\beta^k}\Vert_{q}\drm t\Bigg)^\frac{1}{\beta^k}
\end{align*}
with
\begin{align*}
&C_{p,k}=\frac{\rho_k^{2\beta}}{\rho_{k+1}^2} 
\left[2\beta^k\left(\frac{1\vee\mu(B)^\frac{1}{p}}{\rho_k^2-\rho_{k+1}^2}+\Vert \Deg\Vert_{p}\right)
\Vert \mu^{-1}\Vert_{p}^{q}\right]^{\beta-1}.
\end{align*}
We will iterate this inequality. But first we estimate $ C_{p,k} $. 
By the definition of the radii $\rho_k=\left(1+2^{-k}\left(\sqrt{2}-1\right)\right)\sqrt{\frac{\delta T}{2}}$, we have since $\beta< 2$
\begin{align*}
	\frac{\rho_k^{2\beta}}{\rho_{k+1}^2}=
	\left(\frac{\delta T}{2}\right)^{\beta-1}\frac{(1+2^{-k}(\sqrt 2-1))^2}{(1+2^{-(k+1)}(\sqrt 2-1))^2}\leq \left(\frac{\delta T}{2}\right)^{\beta-1}2^{2\beta}
	\leq 16(T\delta)^{\beta-1} 
\end{align*}
and
$$ \rho_k^2-\rho_{k+1}^2\geq \frac{1}{6}2^{-k}\delta T.$$
These estimates yield, using $\beta<2$ and hence $16\cdot 12\cdot 2^{2k(\beta-1)}\leq 4^{k+5}$
\begin{align*}
C_{p,k}&\leq C_{p,k}'
:=4^{k+5} 
\left[\left(1\vee\mu(B)^\frac{1}{p}+\delta T\Vert \Deg\Vert_{p}\right)
\Vert \mu^{-1}\Vert_{p}^{q}\right]^{\beta-1}.
\end{align*}
Using $ \delta/2\ge \rho_{k+1}^{2} $ and H\"older's inequality, we infer from iterating the above inequality for any $l\in\NN$
\begin{multline*}
\left(\int\limits_{(1-\delta/2)T}^{(1+\delta/2)T}{\Vert v_t^{2\beta^{k+l}}\Vert}_{1}\drm t\right)^\frac{1}{\beta^{k+l}}
\\
\leq \left(2\rho_{k+l}^2\mu(B)^{\frac{1}{q}}\right)^{\frac{1}{\beta^{k+l}}}\left(\frac{1}{2\rho_{k+l}^2}\int\limits_{T-\rho_{k+l}^2}^{T+\rho_{k+l}^2}{\Vert v_t^{2\beta^{k+l}}\Vert}_{q}\drm t\right)^\frac{1}{\beta^{k+l}}
\\
 \leq 
\left(2\rho_{k+l}^2 \mu(B)^{\frac{1}{q}}\right)^{\frac{1}{\beta^{k+l}}}
\prod_{j=1}^l (C_{p,k+j-1}')^\frac{1}{\beta^{k+j}}
\left(\frac{1}{2\rho_k^2}\int\limits_{T-\rho_{k}^2}^{T+\rho_{k}^2}{\Vert v_t^{2\beta^{k}}\Vert}_{q}\drm t\right)^\frac{1}{\beta^k}.
\end{multline*}
Letting $l\to\infty$, the left-hand side converges to the supremum of $v^2$. Since $\rho_k$ is bounded and $\mu(B)<\infty$, the first factor on the right-hand side converge to 1. Applying Jensen's inequality  to the integral on the right-hand side, we obtain
\begin{align*}
\sup_{[(1-\frac{\delta}{2})T,(1+\frac{\delta}{2})T]}\!\!\!\! v^2
&\leq
C
\left(\frac{1}{2\rho_k^2}\int\limits_{T-\rho_{k}^2}^{T+\rho_{k}^2}\sum_{B}\mu\  v_t^{2\beta^kq}\ \drm t\right)^\frac{1}{\beta^kq}\!\!\!\!
\\
&\leq C
\left(\frac{1}{2\delta T}\int\limits_{(1-\delta)T}^{(1+\delta)T}\sum_{B}
\mu\ v_t^{2\beta^kq}\ \drm t\right)^{\frac{1}{\beta^kq}}
\end{align*}
as  $2\rho_k^2\geq  \delta T$ and $ [T-\rho_k^2,T+\rho_k^2]\subset [(1-\delta)T,(1+\delta)T]$, where 
\begin{align*}
C&=
\prod_{j=1}^\infty (C_{p,k+j-1}')^\frac{1}{\beta^{k+j}} 
\\
&=
4^{\sum_{j=1}^\infty\frac{k+j}{\beta^{k+j}}+4\sum_{j=1}^\infty\frac{1}{\beta^{k+j}}} 
\left[\left(1\vee\mu(B)^\frac{1}{p}+\delta T\Vert \Deg\Vert_{p}\right)
\Vert \mu^{-1}\Vert_{p}^{q}\right]^{\frac{\beta-1}{\beta^{k}}\sum_{j=1}^\infty\frac{1}{\beta^j}}\!\!.
\end{align*}
Now the statement follows since $\beta>1$ and using $k\beta^{-k}\leq 1/\ln\beta$, $k\in\NN_0$,
\[\sum_{j=1}^\infty \beta^{-j}=\sum_{j=0}^\infty \beta^{-j}-1=\frac{1}{\beta-1}=\left(1-\frac{1}{\beta}\right)\sum_{j=0}^\infty \frac{j}{\beta^j}=\left(1-\frac{1}{\beta}\right)\sum_{j=1}^\infty \frac{j}{\beta^j},\]
and thus $\sum_{j=1}^\infty\frac{k+j}{\beta^{k+j}}+4\sum_{j=1}^\infty\frac{1}{\beta^{k+j}}\leq ((4+1/\ln\beta)+\beta/(\beta-1))/(\beta-1).$
\end{proof}
%
%
%
%
%
%
%
\section{$\mathbf{\ell^2}$-mean value inequality for solutions}\label{section:l2}
In case of a local Dirichlet form it suffices to prove subsolution estimates and the corresponding Moser steps to obtain an $\ell^2$-mean value inequality for solutions of the heat equation. On graphs, the iteration of subsolutions over space and time comes to a natural stop depending on the radius and the jump size. To obtain an $\ell^2$-mean value inequality for graphs, we iterate estimates for supersolutions in constant space afterwards. This iteration over time alone then yields error terms which capture the unbounded geometry of the graph. We encountered them in the previous section as the error term $ G_{B} $ and, next, we define them in the form in which they  enter the main result, Theorem~\ref{thm:main1sturm}, below.
\\

We set for $  \beta\in(1,1+1/q) $
\[
\theta(r):=\frac{1}{2\beta^{\kappa(r)}}, \qquad  \kappa(r):=\left\lfloor\sqrt{\frac{r}{4S}}-2\right\rfloor.
\]
Define the error-function $\Gamma_x(r)
:=
\Gamma_x(r,p,n,\beta)\geq 0$ by
\[
\Gamma_x(r)=\left[\left(1+r^2D_p\left(r\right)\right)
M_p\left(r\right)^q m\left(B\left(r\right)\right)^{q}\right]^{\theta(r)},
\]
where $ D_{p}(r)=D_{p}(x,R)$ and $ M_{p}(r)=M_{p}(x,r) $ are the $ p $-averages of $ \mathrm{Deg} $ and $ m^{-1} $ over $B(r)= B_{x}(r) $.
Indeed, the first factor including $ D_{p} $ and $ M_{p} $ captures the unboundedness of the geometry, while the second factor is a volume correction term. Note that in the setting we have in mind, the volumes of balls are growing polynomially, such that this second factor will be bounded as both $\theta$ and $\eta$ are exponentially decreasing.
\begin{remark}[Properties of $\theta$ and $\Gamma$]\label{remark:gamma}It is not hard to see that we have 
\[
\theta(r)\asymp \beta^{-\sqrt {r/4S}} =e^{-\gamma\sqrt{r}}
\]
with $\gamma = \ln \beta / \sqrt{4S}>0$.
The function $\Gamma_x$ can easily be estimated if we have some additional control on the volume growth of distance balls, what is a natural assumption in the present setting.
If we assume $V(d,R_1,R_2)$ in $ x $, we get for all $r\in [R_1,R_2]$
\begin{align*}
m(B(r))^{q\theta(r)}
&\lesssim \left(\left(\frac{m(B(R_1))}{R_1^d}\right)r^d\right)^{q\theta(r)}
\lesssim  \left(\frac{m(B(R_1))}{R_1^d}\right)^{q\theta(r)},
\end{align*}
where we used volume doubling in the first line and boundedness of the function $r\mapsto r^{q\theta(r)}$ in the second. This yields
\begin{align*}
\Gamma_x(r)
&\lesssim 
 \sup_{r\in[R_1,R_2]}\left[\left(1+r^2D_p\left(r\right)\right)
M_p\left(r\right)^q\right]^{\theta(r)}\cdot \left(\frac{m(B(R_1))}{R_1^d}\right)^{q\theta(r)}
.
\end{align*}
The first term is clearly bounded in terms of $R_1$ and $R_2$ if $[R_1,R_2]$ is compact. If $R_2=\infty$, we might restrict to $x\in X$ satisfying
$$ r^2D_p(r)M_p(r)^q\lesssim\exp\left(\exp\left(\gamma{\sqrt r}\right)\right) $$ 
on $[R_1,\infty)$, what is satisfied, e.g., if $D_p$ and $M_p$ grow polynomially. Then we get 
\begin{align*}
\Gamma_x(r)
\leq 
C(x,R_1).
\end{align*} 
\end{remark}
 Theorem~\ref{thm:subsolutionsogeneral} and Theorem~\ref{thm:constantballsdavies} applied to balls with the averaging measure imply the following $\ell^2$-mean value inequality.
 
\begin{theorem}\label{thm:l2meanvaluesturm} Let $ x\in X $, $d>0$, $n>2$,  $p\in(1,\infty]
$, $\alpha=1+\tfrac2n$,
\[
\beta=1+\frac{1}{n\vee 2q}, 
\quad\mbox{
and
}\quad
R\geq 
8S\left(\frac{\ln q}{\ln\frac{\alpha}{\beta}}+3\right)^2.
\] 
Assume $SV(R/2,R)$ in $ x $. Then, for  $\tau\in(0,1]$, $T\in\RR$,  
and  non-negative $\Delta_\omega$-solutions $v$ on $[T-R^2,T+R^2]\times B(R)$ we have 
\begin{align*}
\sup_{[T-\tau R^2/8,T+\tau R^2/8]\times B(R/2)}v^2 
\leq 
\frac{C_{d,n,\beta} \Gamma( R/2)^2(1+\tau R^2h(\omega))^{\frac{n}{2}+1}}
{\tau^{\frac{n}{2}+1}R^{2}m(B(R))}
\int\limits_{T-\tau R^2}^{T+ \tau R^2}\sum_{B(R)} m\ v_t^{2}\ \drm t,
\end{align*}
where $C_{d,n,\beta}:=C_\beta C_{d,n} $.
\end{theorem}
\begin{proof} Clearly, $\alpha>\beta$
and we have  
\[
\left(\frac{\alpha}{\beta}\right)^{\kappa(R/2)}=\left(\frac{\alpha}{\beta}\right)^{\left\lfloor\sqrt{\frac{R}{8S}}-2 \right\rfloor}
\geq \left(\frac{\alpha}{\beta}\right)^{\frac{\ln q}{\ln\frac{\alpha}{\beta}}}=q,
\]
such that we have $\alpha^{\kappa(R/2)}\geq \beta^{\kappa(R/2)}q$. 
Since averaged $L^p$-norms are non-decreasing in $p\in[1,\infty]$,
we can use this fact after we applied  Theorem~\ref{thm:constantballsdavies} to the space  $ (B,\mu)=(B(R/2),m_{B(R/2)}) $ and with constants $k=\kappa(R/2)$, $\delta=\tfrac{\tau R^2}{4T}$  to get
\begin{align*}
&	\sup_{[T-\frac{\tau R^2}{8},T+\frac{\tau R^2}{8}]\times B(R/2)} v^2
\\
&\le G\left(\frac{1}{2\tau (R/2)^2}\hspace{-.1cm}\int\limits_{T- \tau R^2/4}^{T+ \tau R^2/4}\sum_{B( R/2)} m_{B( R/2)}\ v_t^{2\beta^{\kappa(R/2)}q}\ \drm t\right)^{\frac{1}{\beta^{\kappa(R/2)}q}}
\\
&\leq G\left(\frac{1}{2\tau (R/2)^2}\hspace{-.1cm}\int\limits_{T-\tau  R^2/4}^{T+\tau R^2/4}\sum_{B( R/2)} m_{B( R/2)} v_t^{2\alpha^{\kappa(R/2)}}\hspace{-.1cm}\drm t\right)^{\frac{1}{\alpha^{\kappa(R/2)}}}\\
&\leq 
\frac{GC_{d,n}(1+\tau R^2h(\omega))^{\frac{n}{2}+1}}
{\tau^{\frac{n}{2}+1}R^{2}}
\int\limits_{T-\tau R^2}^{T+ \tau R^2}\sum_{B(R)} m_{B(R)}\ v_t^{2}\ \drm t,
\end{align*}
where the last estimate follows by
Theorem~\ref{thm:subsolutionsogeneral}  with $\delta=\tau$ which is applicable since  $R\ge  32S $ by assumption.
We obtain the statement since the function $ G=G_{B,\mu}(\tfrac{\tau R^2}{4T},\frac{R}{2},T,\kappa(\frac{R}{2})) $ can be estimated by
\begin{align*}
{G}=  C_\beta \left[\left(1+\frac{\tau R^2}4 D_p\left( {R}/{2}\right)\right)
M_p\left({R}/{2}\right)^qm(B(R/2))\right]^{2\theta(R/2)}\!\!\!\le C_\beta\Gamma_{x}(R/2)^2
\end{align*}
since $ \tau \le 1 $.
This finishes the proof.
\end{proof}

%
%
%
%
%
%
%
%
%

\section{Davies method with $\mathbf{\ell^2}$-mean value inequality}\label{section:davies}

In this section we  derive upper heat kernel bounds from an $ \ell^{2} $ mean value inequality for the sandwiched Laplacian.

For the semigroup $ (P_{t})_{t\ge0} $ of the Laplacian $ \Delta $ on $ \ell^{2}(X,m) $ and $\omega\in \ell^\infty(X)$, we let
\[
P_t^{\omega}:=\euler^{\omega}P_t\euler^{-\omega}.
\]
Then, $(P_t^{\omega})_{t\geq 0}$ is also a semigroup of bounded operators on $\ell^2(X,m)$. 
For $f\in \ell^2(X,m)$, the map $t\mapsto P_t^{\omega} f$ solves the $\ell^2$-Cauchy problem for the operator $\euler^\omega \Delta\euler^{-\omega}$, i.e.,
\[
\frac{d}{d t} P_t^{\omega} f=-\euler^\omega \Delta\euler^{-\omega}P_t^{\omega} f,
\]
since $P_tf$ is the solution of the $\ell^2$-Cauchy problem for $\Delta$. 
For a  bounded linear operator $A\colon \ell^p(X,m)\to \ell^q(X,m)$, we denote the operator norm by $\Vert A\Vert_{p,q}$. 
 
\begin{lemma}\label{thm:daviesabstract}
Let $T\geq 0$. Assume that for $i\in\{1,2\}$ there exist $x_i\in X$, $a_i\leq b_i$, and $\phi\colon \{x_1,x_2\}\times\ell^\infty(X)\to (0,\infty)$ such that for $ f\in \ell^2(X), f\geq 0 $, and $\omega\in\ell^\infty(X)$ 
\[
 \phi(\omega,x_i)^2 (P_T^\omega f)(x_i)^2 
\leq
\int_{a_i}^{b_i}\Vert P_t^\omega f\Vert_2^2\ \drm t.
\]
Then we have  for the heat kernel of $ \Delta $
\begin{multline*}
p_{2T}(x_1,x_2)
\leq 
\inf_{\omega\in \ell^\infty(X)}
\frac{\euler^{\omega(x_2)-\omega(x_1)}}{\phi(\omega,x_1)\phi(-\omega, x_2)}
\left(\int_{a_1}^{b_1}\Vert P_t^{\omega}\Vert_{2,2}^2\drm t\right)^{\frac{1}{2}}
\left(\int_{a_2}^{b_2}\Vert P_t^{\omega}\Vert_{2,2}^2\drm t\right)^{\frac{1}{2}}.
\end{multline*}
\end{lemma}

\begin{proof}
Let $\omega\in \ell^\infty(X)$.
One directly calculates
\begin{align*}
	p_{2T}(x_1,x_2)
	&=
	\euler^{\omega(x_2)-\omega(x_1)}\Vert \Eins_{\{x_1\}} P_{2T}^\omega\Eins_{\{x_2\}}\Vert_{1,\infty}.
\end{align*}
In order to estimate the norm, observe that duality  implies
\[\Vert P_T^{\omega}\Eins_{\{x_2\}} \Vert_{1,2}=\Vert \Eins_{\{x_2\}} P_T^{-\omega}\Vert_{2,\infty}
.\] 
The semigroup property implies
\begin{align*}
	\Vert \Eins_{\{x_1\}}P_{2T}^\omega\Eins_{\{x_2\}}\Vert_{1,\infty}
	&
	\leq 
	\Vert \Eins_{\{x_1\}} P_T^{\omega}\Vert_{2,\infty}
	\Vert P_T^{\omega}\Eins_{\{x_2\}} \Vert_{1,2}
	=
	\Vert \Eins_{\{x_1\}}P_T^{\omega}\Vert_{2,\infty}
	\Vert \Eins_{\{x_2\}}P_T^{-\omega}\Vert_{2,\infty},
\end{align*}
so we are left to estimate the last two norms.
For  $f\in \ell^2(X,m)$,  $f\geq 0$, 
and $\omega\in\ell^\infty(X)$ the assumption readily gives
\begin{align*}
\phi(x_1, \omega)^2 P_T^\omega f(x_1)^2 
&\leq 
\int_{a_1}^{b_1}\Vert P_t^{\omega}\Vert_{2,2}^2\ \drm t\ \Vert f\Vert_2^2.
\end{align*}
Now for general $ f\in \ell^{2}(X,m)$ we decompose $f=f_+-f_-$, where $f_\pm=\max\{0,\pm f\}$. Since $P_{T}$ is positivity preserving and $(c-d)^2\leq c^2+d^2$, $c,d\geq 0$,  the  estimate above yields
\begin{align*}
\phi( x_1,\omega)^2 P_T^\omega f(x_1)^2 
\leq 
 \phi(x_1,\omega)^2 P_T^\omega f_+(x_1)^2 
+\phi(x_1,\omega)^2 P_T^\omega f_-(x_1)^2 
\leq 
\int_{a_1}^{b_1}\Vert P_t^{\omega}\Vert_{2,2}^2\drm t 
\Vert f\Vert_2^2.
\end{align*}
Since $\Vert \Eins_{\{x_1\}}P_T^\omega f \Vert_\infty^2 = P_T^\omega f(x_1)^2 $, we obtain 
\[
\Vert \Eins_{\{x_1\}} P_T^\omega \Vert_{2,\infty}
\leq 
\phi(x_1,\omega)^{-1}\left(\int_{a_1}^{b_1}\Vert P_t^{\omega}\Vert_{2,2}^2\drm t \right)^{\frac{1}{2}}.
\]
An analogous bound holds if we replace $x_1$, $a_1$, $b_1$,  $\omega$ by $x_2$, $a_2$, $b_2$,  $-\omega$ which gives
\begin{align*}
\Vert \Eins_{\{x_2\}}P_T^{-\omega}\Vert_{2,\infty}
\leq 
\phi(x_2,-\omega)^{-1}
\left(\int_{a_2}^{b_2}\Vert P_T^{-\omega}\Vert_{2,2}^2\ \drm t\right)^{\frac{1}{2}}.
\end{align*}
This finishes the proof since $\Vert P_t^{-\omega}\Vert_{2,2}=\Vert P_t^{\omega}\Vert_{2,2}$ because $P_t^{-\omega}$ is the adjoint $P_t^{\omega}$.
\end{proof}

The upper bound in Lemma~\ref{thm:daviesabstract} depends on the best choice of a function in a minimizaton problem. Davies,  \cite{Davies-93} following upon \cite{Davies-87}, chose a certain family of Lipschitz functions and minimized with respect to the Lipschitz constant. We will use this argument as well.\\
To this end, we need an integrated maximum principle for graphs with intrinsic metric from \cite{BauerHuaYau-17}. We say that $\omega \in \cC(X)$ is $\kappa$-Lipschitz if $\omega$ is Lipschitz continuous with respect to the intrinsic metric $\rho$ with Lipschitz constant $\kappa$. Recall that $ S $ denotes the jump size of the intrinsic metric.

\begin{lemma}[{\cite[Lemma~3.3]{BauerHuaYau-17}}]\label{lem:cauchy}Let $\kappa\geq 0$ and $f\in \ell^2(X,m)$. For any $\kappa$-Lipschitz function $\omega\in \cC(X)$, the function 
\begin{equation*}
\Phi:[0,\infty)\to[0,\infty),\quad \ t\mapsto \exp\left(2\Lambda t -\frac{2}{S^2}\left(\cosh\left(\kappa S\right)-1\right)t\right)\ \Vert \euler^{\omega}P_tf\Vert_2^2
\end{equation*}
is non-increasing.
\end{lemma}

Note that the constant $2S^{-2}(\cosh(\kappa S)-1)$ is in fact an upper bound for $h(\omega)$ defined in Section~\ref{section:tas} if $\omega$ is $\kappa$-Lipschitz as shown in the proof of Theorem~\ref{theorem:daviesabstractgraph}. Moreover, it is the reason for the specific form of the function $\zeta_S$ below.\\
The following theorem provides a heat kernel upper bound for a fixed time on subsets of a graph where an $\ell^2$-mean value inequality holds. To this end, recall the function $\zeta_S$ defined for $ r\geq 0$, $t>0$ by
\[
\zeta_S(r,t)=\frac{1}{S^2}\left(r S \arsinh\left(\frac{r S}{t}\right)+t-\sqrt{t^2+r^2S^2}\right).
\]

\begin{theorem}\label{theorem:daviesabstractgraph}Let $T>0$,
$Y\subset X$,
 $a, b\colon Y \to [0,\infty)$, $a\leq b$, $\phi\colon Y\times [0,\infty)\to[0,\infty)$ with $\phi(x,\cdot)$ non-increasing, $x\in Y$, such that 
for all $x\in Y$ and $\omega\in\ell^\infty(X)$ we have
\[
\phi(x,h(\omega))^{2}(P_T^\omega f)^2(x)\leq  \int_{a(x)}^{b(x)}\Vert P_t^\omega f\Vert_2^2\ \drm t,\quad f\in \ell^2(X), f\geq 0.
\]
Then we have 
for all $x,y\in Y$
\begin{align*}
p_{2T}(x,y)
&\leq 
\frac
{
(b(x)-a(x))^{\frac{1}{2}}(b(y)-a(y))^{\frac{1}{2}}
\exp\left(\frac{b(x)+b(y)-2T}{2}\sigma(\rho(x,y),2T)\right)
}
{\phi\big(x,\sigma(\rho(x,y),2T)\big)\phi\big(y,\sigma(\rho(x,y),2T)\big)}
\\
&\quad \cdot\exp\big(-\Lambda(a(x)+a(y))-\zeta_S(\rho(x,y),2T)\big),
\end{align*}
where
\[
\sigma(r,t):=2 S^{-2}\left(\sqrt{1+\frac{r^2S^2}{t^2}}-1\right).
\]
\end{theorem}

\begin{proof} For  $\kappa\geq 0$, let
\[E_\kappa:=\{\omega\in \ell^\infty(X)\colon \omega\  \mbox{ is  } \kappa\text{-Lipschitz}\}\subset \ell^\infty(X).\]
Lemma~\ref{thm:daviesabstract} implies for all $x,y\in Y$
\[
p_{2T}(x,y)
\leq 
\inf_{\omega\in \ell^\infty(X)}
I_1\cdot I_2
\leq \inf_{\kappa>0}\inf_{\omega\in E_\kappa}I_1\cdot I_2,
\]
with
\begin{align*}
I_1:=\frac{\euler^{\omega(y)-\omega(x)}}{\phi(x,h(\omega))\phi(y,h(-\omega))},\quad
I_2:=\left(\int_{a(x)}^{b(x)}\Vert P_t^{\omega}\Vert_{2,2}^2\drm t\right)^{\frac{1}{2}}
\!\!\!
\left(\int_{a(y)}^{b(y)}\Vert P_t^{\omega}\Vert_{2,2}^2\drm t\right)^{\frac{1}{2}}\!\!\!.
\end{align*}

To estimate $I_1$, we basically use Davies' idea from \cite[Theorem~10]{Davies-93}. We choose
\[
\nu:=\min\{\rho(x,\cdot),\rho(x,y)\}
\]
and for $\kappa\geq 0$
\[
\omega:=-\kappa \nu.
\]
We clearly have $\nu\in E_1$, $\omega\in E_\kappa$, and 
\begin{align*}
\omega(y)-\omega(x)=-\kappa \rho(x,y).
\end{align*}
With these choices we bound $h(\omega)$. First, a direct calculation gives 
\begin{align*}
\vert \nabla_{xy}\euler^{\omega}\nabla_{xy}\euler^{-\omega}\vert=2(\cosh(\kappa \rho(x,y))-1).
\end{align*}
Since the map
\[
t\ \mapsto \ \frac{1}{t^2}\big(\cosh(\kappa t)-1\big), \quad t>0
\]
is monotonically increasing,
we get for $b(x,y)>0$ (cf.~\cite[p.~1434]{BauerHuaYau-17})
\[
\cosh(\kappa \rho(x,y))-1\leq \frac{\rho(x,y)^2}{S^2}(\cosh(\kappa S)-1).
\]
Plugging in these estimates and using that $\rho$ is intrinsic yields 
\begin{align*}
h(\omega)&=\sup_{x\in X}\sum_{x\in X}\frac{b(x,y)}{m(x)}\vert \nabla_{xy}\euler^{\omega}\nabla_{xy}\euler^{-\omega}\vert
=2 \sup_{x\in X}\sum_{x\in X}\frac{b(x,y)}{m(x)}(\cosh(\kappa\rho(x,y))-1)
\\
&\leq 2 \sup_{x\in X}\sum_{x\in X}\frac{b(x,y)}{m(x)}\frac{\rho(x,y)^2}{S^2}(\cosh(\kappa S)-1)
\leq 2 S^{-2}(\cosh(\kappa S)-1)=:\gamma(\kappa S).
\end{align*}
Using $h(\omega)=h(-\omega)$ and the monotonicity properties of $\phi$, we obtain 
\begin{align*}
I_1
=\frac{\euler^{\omega(y)-\omega(x)}}{\phi(x,h(\omega))\phi(y,h(-\omega))}
\leq \frac{\euler^{-\kappa\rho(x,y)}}{\phi(x, \gamma(\kappa S))\phi(y, \gamma(\kappa S))}.
\end{align*}

Now, we estimate $I_2$. Let $f\in \ell^2(X)$ with $ \|f\|_{2}=1 $. Lemma~\ref{lem:cauchy} applied to the function $ e^{-\omega }f $ gives $ \Phi(t)\leq \Phi(0)=1 $, $t\geq 0$. Hence, recalling $\gamma(\kappa S)= 2S^{-2}\left(\cosh\left(\kappa S\right)-1\right)$, we obtain for all $t\geq 0$ 
\begin{align*}
\Vert P_t^{\omega} f\Vert_2^2
&=\Vert \euler^{\omega} P_t \euler^{-\omega}f\Vert_2^2
=\Phi(t)\euler^{-2\Lambda t+2S^{-2}\left(\cosh\left(\kappa S\right)-1\right)t}
\leq 
\euler^{-2\Lambda t+\gamma(\kappa S)t}.
\end{align*}
Since $\cosh\geq 1$, we get for $z\in\{x,y\}$
\begin{align*}
\int_{a(z)}^{b(z)}\Vert P_t^{\omega}\Vert_{2,2}^2 \drm t
&\leq 
\int_{a(z)}^{b(z)}\euler^{-2\Lambda t+\gamma(\kappa S)t} \drm t
\leq (b(z)-a(z))\euler^{-2\Lambda a(z)+\gamma(\kappa S)b(z)}.
\end{align*}
Hence,
\begin{align*}
I_2\leq (b(x)-a(x))^{\frac{1}{2}}(b(y)-a(y))^{\frac{1}{2}}
 \exp\left(-\Lambda(a(x)+a(y))+\frac{b(x)+b(y)}{2}\gamma(\kappa S)\right).
\end{align*}
Now we put everything together.
Set
\[
r:=\rho(x,y), \quad B:=b(x)+b(y),
\]
\[
F(\theta)
:=\inf_{\eta>0} f(\eta, \theta), \quad f(\eta,\theta):=\left(-\eta+\frac{1}{\theta}(\cosh(\eta)-1)\right).
\]
The value $\eta_\theta=\arsinh(\theta)=:\kappa_\theta S$ is the minimum of the function $\eta\mapsto f(\eta, \theta)$, $\theta>0$. 
Plugging in the minimizer $\eta_0:=\eta_{\theta_0}$ for the choice $\theta=\theta_0:=rS/2T$ yields
\begin{align*}
&\inf_{\kappa>0}\inf_{\omega\in E_\kappa} I_1\cdot I_2 
\\
&\leq 
\inf_{\kappa>0}\!\frac{(b(x)-a(x))^{\frac{1}{2}}(b(y)-a(y))^{\frac{1}{2}}}{\phi(x,\gamma(\kappa S))\phi(y,\gamma(\kappa S))}
 \exp\left(\!\!-\Lambda(a(x)+a(y))\!+\!\left(\!\!-\kappa r+\frac{2T+B-2T}{2}\gamma(\kappa S)\right)\!\!\right)
 \\
 &
=
\inf_{\kappa>0}\!\frac{(b(x)-a(x))^{\frac{1}{2}}(b(y)-a(y))^{\frac{1}{2}}}{\phi(x,\gamma(\kappa S))\phi(y,\gamma(\kappa S))}
 \exp\left(\!\!-\Lambda(a(x)+a(y))\!+\!\frac{r}{S}f\left(\kappa S,\frac{rS}{2T}\!\right)\!+\!\frac{B-2T}{2}\gamma(\kappa S)\!\!\right)
\\
&\leq 
\frac{(b(x)-a(x))^{\frac{1}{2}}(b(y)-a(y))^{\frac{1}{2}}}{\phi(x,\gamma(\eta_0))\phi(y,\gamma(\eta_0))}
\exp\left(-\Lambda(a(x)+a(y))+\frac{r}{S}F\left(\frac{rS}{2T}\right)+\frac{B-2T}{2}\gamma(\eta_0)\right).
\end{align*}
We are left with simplifying the remaining terms. The hyperbolic Pythagorean theorem yields $\cosh(\arsinh(\theta))=\sqrt{1+\theta^2}$. Hence, 
\[
\gamma(\eta_0)=2S^{-2}\left(\cosh\left(\arsinh\left(\frac{rS}{2T}\right)\right)-1\right)=2S^{-2}\left(\sqrt{1+\frac{r^2S^2}{4T^2}}-1\right)=\sigma(r,2T).
\]
Moreover, as $\eta_\theta=\arsinh(\theta)$ minimizes $f(\cdot, \theta)$, we get 
\begin{align*}
F(\theta)
=\frac{1}{\theta}\left(\sqrt{1+\theta^2}-1\right)-\arsinh\left(\theta\right).
\end{align*}
Since
\[
\frac{r}{S}F\left(\frac{rS}{2T}\right)=\frac{r}{S}\frac{2T}{rS}\left(\sqrt{1+{\frac{rS}{2T}}^2}-1\right)-\frac{r}{S}\arsinh\left(\frac{rS}{2T}\right)=-\zeta_S(r, 2T),
\]
the claim follows.
\end{proof}

\section{Main result: Gaussian upper bounds}\label{section:main}

Recall that for $x\in X$ the error function $\Gamma_x\geq 0$ is defined by 
\[
\Gamma_x(r)=\left[\left(1+r^2D_p\left(r\right)\right)
M_p\left(r\right)^qm\left(B\left(r\right)\right)^{q}\right]^{\theta(r)},
\]
where $\theta(r)\asymp \euler^{-\gamma\sqrt{r}}$. The following theorem is our main result.

\begin{theorem}\label{thm:main1sturm}
Let $I$ be an arbitrary index set, $x_i\in X$, $d_i>0$, $n_i>2$, $p_i\in(1,\infty]$, $q_i$ the H\"older conjugate of $p_i$,
\[
\alpha_i=1+\frac2{n_i}, \quad \beta_i=1+\frac{1}{n_i\vee 2q_i}, \quad R_{0,i}:=8S\left(\frac{\ln q_i}{\ln\frac{\alpha_i}{\beta_i}}+3\right)^2,
\]
and $R_i\geq 2 r_i\geq 2 R_{0,i}$. Assume 
 that for all $i\in I$, $x_i$ satisfies $SV(r_i,R_i)$. Then for all $i,j\in I$
 and $t\geq 8\max\{r_i^2,r_j^2\}$ we have
\begin{multline*}
p_{t}(x_i,x_j)
\leq 
C
\Gamma_{x_i}(\tau_i)\Gamma_{x_j}(\tau_j)
\\
\cdot
\frac{\left(1\vee S^{-2}\left(\sqrt{t^2+\rho_{ij}^2S^2}-t\right)\right)^{\frac{n_{ij}}{2}}}
 {\sqrt{m(B_{x_i}(\sqrt {t}\wedge R_i))m(B_{x_j}(\sqrt {t}\wedge R_j))}}
\euler^{-\Lambda (t-\frac{1}{2}(t\wedge R_i^2+t\wedge R_j^2))-\zeta_S\left(\rho_{ij},t\right)},
\end{multline*}
where  
\[
\rho_{ij}:=\rho(x_i,x_j),\quad n_{ij}=\frac{n_i+n_j}{2}, \quad
\tau_k=\sqrt {\frac{t}{8}} \wedge \frac{R_k}2, \quad k\in \{i,j\},
\]
and 
\[
C=2^{3+n_{ij}+d_i+d_j}\euler C_{D,i}C_{D,j}\sqrt{C_{d_i,n_i,\beta_i}C_{d_j,n_j,\beta_j}},
\]
where $C_{D,k}$ is the doubling constants of Definition~\ref{def:sobvol}(ii), and   $C_{d_k,n_k,\beta_k}$ is given by Theorem~\ref{thm:l2meanvaluesturm}, $k\in\{i,j\}$.
\end{theorem}
\begin{proof}
Let $i\in I$. Since $x_i$ satisfies $SV(r_i,R_i)$, it satisfies in particular $SV(R/2,R)$ for all $ R\in[2r_{i},R_{i}] $. 
If $f\in \ell^2(X)$, $f\geq 0$, then the function $(t,x)\mapsto P_t f(x)$ is a solution on $[0,\infty)\times X$. Hence, Theorem~\ref{thm:l2meanvaluesturm} implies for all $\delta\in(0,1]$, $T\geq \delta R^2$, and $\omega\in\ell^\infty(X)$
\begin{align*}
P_T^\omega f(x_i)^2
&\leq 
\sup_{(t,x)\in [T-\delta R^2/8,T+\delta R^2/8]\times B_{x_i}(R/2)}(P_t^\omega f)(x)^2 
\\
&
\leq 
\frac{C_i \Gamma_{x_i}(R/2)^2(1+\delta R^2h(\omega))^{\frac{n}{2}+1}}{\delta^{\frac{n_i}2+1}R^2m(B_{x_i}(R))}
\int\limits_{T-\delta R^2}^{T+\delta R^2}\Vert\Eins_{B_{x_i}(R)} P_t^\omega f\Vert_2^{2}\drm t,
\end{align*}
where we 
 set $C_i:=C_{d_i,n_i,\beta_i}$.
For $T>0$, set
\[
r_i':= \left(\sqrt T \wedge R_i\right)\vee 2r_i \in [2r_i,R_i].
\]
If $T\geq 4 r_i^2$, then we have $T-\delta r_i'^2\geq 0$. Set
$Y=\{x_i,x_j\}$ and for $k\in\{i,j\}$
\[
a(x_k)=T-\delta r_k'^2,\quad b(x_k)=T+\delta r_k'^2,\quad
 r(x_k)=r_k',
\]
and $\phi(x_k, h(\omega))>  0$ given by
\[
\phi(x_k,h(\omega))^{-2}=\frac{C_k \Gamma_{x_k} (r_k'/2)^2}{\delta^{\frac{n_i}2+1}r_k'^{2}
 m(B_{x_k}(r_k'))}(1+\delta r_k'^2h(\omega))^{\frac{n}{2}+1}.
\]
Theorem~\ref{theorem:daviesabstractgraph} yields
for $T\geq 4\max\{r_i^2, r_j^2\}$, and $ t=2T $
 the estimate

\begin{align*}
&p_{t}(x_i,x_j)=p_{2T}(x_i,x_j)
\\&
\leq \frac
{
(b(x_i)-a(x_i))^{\frac{1}{2}}(b(x_j)-a(x_j))^{\frac{1}{2}}
\exp\left(\frac{b(x_i)+b(x_j)-2T}{2}\sigma(\rho(x_i,x_j),2T)\right)
}
{\phi\left(x_i,\sigma(\rho(x_i,x_j),2T)\right)\phi\left(x_j,\sigma(\rho(x_i,x_j),2T)\right)}
\\
&\quad \cdot\exp\left(-\Lambda(a(x_i)+a(x_j))-\zeta_S(\rho(x_i,x_j),2T)\right)
\\
&
=
\frac{C'
 \Gamma_{x_i} (r_i'/2)\Gamma_{x_j}(r_j'/2)\delta^{\frac{-n_{ij}}{2}}}
 {\sqrt{m(B_{x_i}(r_i'))m(B_{x_j}(r_j'))}} 
 \left(1+\delta r_i'^2\sigma(\rho_{ij},t)\right)^{\frac{n}{4}+\frac12}
 \left(1+\delta r_j'^2\sigma(\rho_{ij},t)\right)^{\frac{n}{4}+\frac12}
 \\
 &\quad\cdot
 \exp\left(\frac{\delta(r_i'^2+r_j'^2)}{2S^2}\sigma(\rho_{ij},t)\right)
\exp\left(-\Lambda (t-\delta (r_i'^2+ r_j'^2))-\zeta_S(\rho_{ij},t)\right),
\end{align*}

where $\sigma(r,t)=2 S^{-2}\left(\sqrt{1+\frac{r^2S^2}{t^2}}-1\right) $, and $\rho_{ij}=\rho(x_i,x_j)$, $n_{ij}=\frac{1}{2}(n_i+n_j)$, and  $C'=2\sqrt{C_iC_j}$.
The rest of the proof is devoted to the estimation the remaining terms.\\

First observe that $ r'_{k}/2=\tau_{k} $ which takes care of the $ \Gamma $ terms. 
For the volume terms we use $ t\geq 8\max\{r_i^2,r_j^2\}$, such that  we have $r_k'= \sqrt{t/2} \wedge R_k$, $k\in\{1,2\}$. 
Thus,  volume doubling leads to
\begin{align*}
\frac{1}{m(B_{x_k}(r_k'))}
=
\frac{1}{m(B_{x_k}(\sqrt {t/2}\wedge R_k))}
\leq \frac{2^{d_k} C_{D,k}}{m(B_{x_k}(\sqrt {t}\wedge R_k))}
.
\end{align*}

Now, we estimate the last exponential term. 
Since  $ t\geq 8\max\{r_i^2,r_j^2\}$, we have 
\begin{align*}\exp\left(-\Lambda(t-\delta(r_i'^2+r_j'^2)\right)
&\leq 
\exp\left(-\Lambda(t-\delta((t/2)\wedge R_i^2+(t/2)\wedge R_j^2) )\right)
\\
&\leq \exp\left(-\Lambda(t-\delta(t\wedge R_i^2+t\wedge R_j^2) )\right).
\end{align*}
Noting that $r_k'\leq \sqrt t$ and choosing 
\[
\delta=\frac{1}{2}\wedge \frac{1}{t\sigma(\rho_{ij},t)}
=
\frac{1}{2}\left(1\wedge \frac{S^2}{\sqrt{t^2+\rho_{ij}^2S^2}-t}\right)
\]
leads to the result.
\end{proof}

In the remaining part of this section we prove all the theorems presented in the introduction.

Theorem~\ref{thm:norm}, which treats the normalized Laplacian follows immediately from additional properties of the normalization. In the next lemma we observe that the Sobolev inequality gives control on the function $\Gamma$, which immediately implies Theorem~\ref{thm:norm}. 
\begin{lemma}\label{lemma:gammanormalized} Let $n>2$. 
There is a constant $ C(n)>0 $ such that if for the normalized Laplacian, $S(n,R)$  holds in $ x $  for $ R\ge0 $, then
\[
\Gamma_x(R) \leq C(n).
\]
\end{lemma}
\begin{proof}
For $y\in B_{x}(R)$ the assumption
$S(n,R)$ in $ x $ yields with $u=\Eins_{\{y\}}$
\begin{align*}
\frac{m(B(R))^{\frac{2}{n}}}{C_SR^2}m(y)^{\frac{n-2}{n}}
&\leq \Vert \vert \nabla\Eins_{\{y\}}\vert\Vert_2^2+\frac{1}{R^2}m(y)
\leq \left(2+\frac{1}{R^2}\right)m(y),
\end{align*}
Since this holds for all $ y\in B_{x}(R) $ we obtain
\[
 \frac{1}{m} \leq 
\frac{C_S^{\frac{n}{2}}\left(2R^2+1\right)^{\frac{n}{2}}}{m(B(R))}.
\]
Hence, plugging in this lower bound for $m(x)$, $x\in B(R)$, we obtain
\[
M_p(R)=\left(\frac{1}{m(B(R))}\sum_{B(R)}m\ \frac{1}{m^p}\right)^{\frac{1}{p}}
\leq 
\frac{C_S^{\frac{n}{2}}\left(2R^2+1\right)^{\frac{n}{2}}}{m(B(R))}.
\]
Since $\Deg=1$ on $ X$, and hence $D_p(R)=1$, this yields 
\begin{align*}
\Gamma_x(R)
&=\!\left[\left(1+R^2D_p(R)\right)
M_p\left(R\right)^qm\left(B\left(R\right)\right)^{q}\right]^{\theta(R)}
\\
&\leq\left[\left(1+R^2\right)
\left(C_S^{\frac{n}{2}}\left(2R^2+1\right)^{\frac{n}{2}}\right)^q\right]^{\theta(R)}\hspace{-.2cm},
\end{align*}
and for any $p\in(1,\infty]$, this upper bound is a bounded function in $R$. The particular choice $p=\infty$ leads to the claim.
\end{proof}

Next, we prove Theorem~\ref{thm:counting} which is the case  $\inf_X m>0$.

\begin{proof}[Proof of Theorem~\ref{thm:counting}]If   $\inf_X m>0$, then we have $M_p\leq 1/(\inf_X m)$ and for all $R\in[r_x,R_x]$ we get by volume doubling
\[
m\left(B\left(R\right)\right)^{q\theta(R)}
\leq \left(C_{D}R^d \frac{m(B_x(r_x))}{r_x^d}\right)^{q\theta(R)}
\leq C(\mu_x(r_x))^{\theta(R)},
\] such that
\begin{align*}
\Gamma_{x_i}(R)
&\leq \left[\left(1+R^2D_p\left(R\right)\right)
\right]^{\theta(R)}m\left(B\left(R\right)\right)^{q\theta(R)}
\\
&
\leq C\left[\mu_x({r_x})\left(1+R^2D_p\left(R\right)\right)
\right]^{\theta(R)}.\qedhere
\end{align*}
\end{proof}

\begin{proof}[Proof of Theorem~\ref{thm:infty}]The assumption
$SV(R_1,\infty)$ implies $SV(R_2,\infty)$ for the radius $R_2:=\max\{R_{0},R_1\}$, where $R_0:=8S(\ln q/\ln (\alpha/\beta)+3)^2$. The result can be derived from Theorem~\ref{thm:main1sturm}.

 The obtained heat kernel bound needs to be estimated further. Volume doubling yields that balls grow  at most polynomially. This together with our assumptions on $M_p$ and $D_p$, we have
\[\Gamma_x(R)\leq C e^{e^{\gamma\sqrt{r}}\theta(r)} m\left(B\left(R\right)\right)^{q\theta(R)}\le C'\mu_x(R_1)^{\theta(R)},\] which is bounded as $ \theta(r)\asymp e^{-\gamma\sqrt{r}} $. Finally, polynomial volume growth yields $\Lambda=0$ by the Brooks-Sturm theorem for unbounded graph Laplacians, \cite{HaeselerKW-13}.
\end{proof}

{\noindent\textbf{Acknowledgements.} The authors acknowledge the financial support of the DFG. Moreover, we thank Alexander Grigor'yan for a valuable discussion and remarks which led to an improvement of the results. }

\scriptsize

\bibliographystyle{alpha}
\newcommand{\etalchar}[1]{$^{#1}$}

\end{document}